\documentclass{amsart}
\usepackage[all]{xy}
\author{Louis de
  Thanhoffer de V\"olcsey} 
\author{Michel Van den Bergh}
\email{louicious@mac.com, michel.vandenbergh@uhasselt.be}
\address{Departement WNI,  Universiteit Hasselt,
 3590 Diepenbeek, Belgium.}
\thanks{The second author is a senior researcher at the FWO}
\usepackage{amssymb}
\usepackage{verbatim}
\numberwithin{equation}{section}
\title{Calabi-Yau Deformations and Negative Cyclic Homology}

\theoremstyle{plain}
\newtheorem{theorem}{Theorem}[section]
\newtheorem{corollary}[theorem]{Corollary}
\newtheorem{proposition}[theorem]{Proposition}
\newtheorem{lemma}[theorem]{Lemma}

\theoremstyle{definition}
\newtheorem{definition}[theorem]{Definition}

\theoremstyle{remark}
\newtheorem{remark}[theorem]{Remark}
\newtheorem{example}[theorem]{Example}

\renewcommand{\b}{\bullet}
\newcommand{\ds}{\oplus}
\newcommand{\tr}{\otimes}

\newcommand{\mor}{\longrightarrow}
\newcommand{\fun}{\mapsto}

\newcommand{\f}[1]{\ensuremath{\mathfrak{#1}}}

\newcommand{\RHom}{\operatorname{RHom}}
\newcommand{\Hom}{\operatorname{Hom}}
\newcommand{\Nilp}{\operatorname{Nilp}}
\newcommand{\Gd}{\operatorname{Gd}}
\newcommand{\MC}{\operatorname{MC}}
\newcommand{\CC}{\operatorname{CC}}
\newcommand{\C}{\operatorname{C}}
\newcommand{\HC}{\operatorname{HC}}
\newcommand{\HH}{\operatorname{HH}}
\newcommand{\poly}{\operatorname{poly}}
\newcommand{\divu}{\operatorname{div}}
\newcommand{\MCcal}{\operatorname{\mathcal{MC}}}

\newcommand{\Def}{\operatorname{Def}}
\newcommand{\Ext}{\operatorname{Ext}}
\newcommand{\Ad}{\operatorname{ad}}
\newcommand{\Id}{\operatorname{Id}}
\newcommand{\ad}{\operatorname{ad}}

\renewcommand{\div}{\operatorname{div}}
\begin{document}
\begin{abstract}
  In this paper we relate the deformation theory of Ginzburg
  Calabi-Yau algebras to negative cyclic homology. We do this by
  exhibiting a DG-Lie algebra that controls this deformation theory
  and whose homology is negative cyclic homology. We show that the
  bracket induced on negative cyclic homology coincides with Menichi's
  string topology bracket.  We show in addition that the obstructions
  against deforming Calabi-Yau algebras are annihilated by the map to
  periodic cyclic homology.

In the commutative case we show that our DG-Lie algebra is homotopy
equivalent to $(T^{\operatorname{poly}}[[u]],-u\operatorname{div})$.
\end{abstract}

\maketitle

\setcounter{tocdepth}{1}
\tableofcontents
\section{Introduction}
Throughout $k$ is a field of characteristic zero. In this paper we
discuss the deformation theory of Calabi-Yau $k$-algebras in the sense of
Ginzburg \cite{Gi}. Recall that a $k$-algebra $A$ is  $d$-Calabi-Yau
if it is perfect as $A^e$-module and there is an isomorphism\footnote{This isomorphism
is sometimes required to satisfy a certain symmetry condition but this happens
to be automatically satisfied. See \cite[App.\ C]{VdBsuper}.} in $D(A^e)$
\begin{equation}
\label{ref-1.1-0}
\eta:\RHom_{A^e}(A,A^e)\mor \Sigma^{-d}A
\end{equation}
In the rest of this introduction we fix a $d$-CY algebra $A$. Here and
throughout the paper we take the point of view that $\eta$ is part of
the structure of 
$A$. 

\medskip

The definition of a $d$-Calabi-Yau algebra can be ``relativized'' without
any difficulty. Hence there is an associated deformation theory. Our
first result in this paper is the construction of a DG-Lie algebra
which controls this deformation theory. 

To be more precise: let $\Nilp$ be the category of commutative, finite
dimensional, local $k$-algebras $(R,m)$ such that $R/m=k$. For $R\in
\Nilp$ let $\Def_{A,\eta}(R)$ be the category of $R$-algebras $B$ which are
 $d$-Calabi-Yau (with respect to $R$) and which are in addition 
equipped with an isomorphism $B\otimes_R k\cong A$ respecting~$\eta$. We view $\Def_{A,\eta}$ as a pseudo-functor from $\Nilp$ to the
category of groupoids $\Gd$.

For a nilpotent DG-Lie algebra $\mathfrak{h}$ let
$\MC(\frak{h})$ be the groupoid of solutions to the Maurer-Cartan
equation in $\mathfrak{h}$ (see \S\ref{ref-7-32}).  For an arbitrary DG-Lie algebra 
$\mathfrak{g}$ we have an
associated ``deformation functor''
\[
\MCcal(\mathfrak{g}):\Nilp\mor \Gd:(R,m)\mapsto \MC(\mathfrak{g}\otimes_k m)
\]
In this paper we introduce a DG-Lie algebra $\mathfrak{D}^\bullet(A,\eta)$
(see \S\ref{ref-8-36}) which controls the deformation theory of $(A,\eta).$\footnote{In fact this is slightly imprecise as $\mathfrak{D}^\bullet(A,\eta)$
is only determined up to a non-unique isomorphism.
The actual definition of $\mathfrak{D}^\bullet(A,\eta)$ depends on the lift of $\eta$ to an explicit cycle in a suitable complex
but we will ignore this subtlety in the introduction.}
\begin{theorem} (a combination of Prop.\ \ref{ref-6.2-30} and Thm \ref{ref-8.1-37} below) 
There is a
 morphism of pseudo-functors
  $\pi:\MCcal(\mathfrak{D}^\bullet(A,\eta)) \mor \Def_{A,\eta}$ which when evaluated on an arbitrary
  $R\in \Nilp$ is essentially surjective on objects and surjective on
  morphisms.
\end{theorem}
We obtain in particular for $R\in \Nilp$ a bijection between
$\MC(\mathfrak{D}^\bullet(A)\otimes_k m)/{\cong}$ and
$\Def_{A,\eta}(R)/{\cong}$.  In this sense the deformation theory of $A$
is controlled by the DG-Lie algebra $\mathfrak{D}^\bullet(A,\eta)$.

\medskip

$\mathfrak{D}^\bullet(A,\eta)$ is constructed as a twisted semi-direct product
of the Hochschild cochain complex with the negative cyclic chain complex of $A$. So the construction is
similar in spirit to \cite{TerTrad} which treats finite dimensional $A_\infty$-algebras
with a non-degenerate inner product. However our algebras are not finite dimensional
and they do not carry an inner product.

The construction of $\mathfrak{D}^\bullet(A,\eta)$  yields a  
morphism
\[
\phi:\mathfrak{D}^\bullet(A,\eta) \mor \bar{\mathfrak{C}}^\bullet(A)
\]
where $\bar{\mathfrak{C}}^\bullet(A)$ is the (shifted) Hochschild cochain complex of $A$.
As is well-known, $\bar{\mathfrak{C}}^\bullet(A)$ controls the deformation theory of $A$ as  algebra. 
The morphism $\phi$
corresponds to ``forgetting
$\eta$'' as is explained in \S\ref{ref-9-47}.

\medskip

The next result is the construction of  an explicit quasi-isomorphism of complexes 
\begin{equation}
\label{ref-1.2-1}
\mathfrak{D}^\bullet(A,\eta)\xrightarrow{\cong} \Sigma^{-d+1}\CC^-_\bullet(A)
\end{equation}
between $\mathfrak{D}^\bullet(A,\eta)$ and the shifted negative cyclic complex
$\CC^-_\bullet(A)$. As a result we obtain the following information about
the deformation theory of Calabi-Yau algebras. 
\begin{theorem} 
\begin{enumerate}
\item The tangent space to the deformation space of a $d$-Calabi-Yau
  algebra is $\HC^-_{d-2}(A)$. 
\item The obstructions against deforming a
  $d$-Calabi-Yau algebra are in the kernel of the canonical map from $\HC^-_{d-3}(A)$
to periodic cyclic homology $\HC^{\text{per}}_{d-3}(A)$.
\end{enumerate}
\end{theorem}
The first statement is a formal consequence of \eqref{ref-1.2-1}. The
second statement is Theorem \ref{ref-12.1-70} below.  It follows in
particular that if $\HC_{d-3}^-(A)\rightarrow
\HC^{\text{per}}_{d-3}(A)$ is injective then the deformation theory of
$A$ as Calabi-Yau algebra is unobstructed. This happens for example if
$d\le 3$ (see Corollary \ref{ref-12.5-74} and Lemma \ref{ref-12.6-75}
below).

\medskip

Our next result is the description of the Lie bracket on $\HC_\bullet^-(A)$ induced by \eqref{ref-1.2-1}:
\begin{theorem} (Theorem \ref{ref-10.2-51} below)
The Lie bracket on negative cyclic homology induced by \eqref{ref-1.2-1} is
the ``string topology'' Lie bracket introduced in \cite{Menichi} by Menichi.
\end{theorem}
Let us now specialize to the case where $A$ is commutative. Let
$T^{\poly,\bullet}(A)$ be the Lie algebra of poly-vector fields on $A$. Then $\eta$ in \eqref{ref-1.1-0}
may be interpreted as a volume form (see \S\ref{ref-11-59} below).
Let $\divu$ be the divergence
operator on $T^{\poly,\bullet}(A)$  associated to $\eta$.  
Using
Willwacher's ``formality for cyclic chains'' \cite{Willwacher1} 
(see also \cite{Dolgushev,Shoikhet2,Tsygan})
we
show that there is an
isomorphism
\[
(T^{\poly,\bullet}[[u]],-u\divu)\xrightarrow{\cong} \mathfrak{D}^\bullet(A,\eta)
\qquad \text{($|u|=2$)}
\]
in the homotopy category of DG-Lie algebras which fits in a commutative diagram
\begin{equation}
\label{ref-1.3-2}
\xymatrix{
(T^{\poly,\bullet}(A)[[u]],-u\divu)\ar[r]^-\cong\ar[d]_{u\mapsto 0}& 
\mathfrak{D}^\bullet(A,\eta)\ar[d]^\phi\\
T^{\poly,\bullet}(A)\ar[r]_\cong&\bar{\mathfrak{C}}^\bullet(A)
}
\end{equation}
The lower arrow is a globalized version of  Kontsevich's formality
quasi-isomor\-phism~\cite{Ko3}.  

\medskip

We conclude that the deformation theory of a commutative Calabi-Yau algebra $A$ is controlled by the DG-Lie algebra
$(T^{\poly,\bullet}(A)[[u]],-u\divu)$.
We obtain in particular an explanation of 
 Dolgushev's result \cite{Dolgushev2} that 
the Kontsevich $\ast$-product associated to a divergence free Poisson
bracket is Calabi-Yau.\footnote{This is a special
case Dolgushev's result. Dolgushev does not restrict to the Calabi-Yau case.}  To see this it suffices to apply
the $\operatorname{MC}$-functor to \eqref{ref-1.3-2}.
\begin{remark} The referee points out that the first work on deformation theory of an algebra
with some sort of  CY property is due to Connes, Flato and Sternheimer~\cite{CFS}.
In the topological setting they considered algebras with a trace such that the scalar
product $\operatorname{tr}(ab)$ is nondegenerate and they related deformations of such 
algebras to  cyclic cohomology. Kontsevich made a corresponding ``formality for cyclic cochains'' conjecture
which was  proven by Calaque and Willwacher  in \cite{Willwacher2}.

The results in this paper are similar in spirit but technically quite different. We depend
in fact on Willwacher's result on formality for cyclic \emph{chains} \cite{Willwacher1} rather than cochains. 
\end{remark}
\section{Acknowledgement}
We like to thank Damien Calaque and Boris Tsygan for help with references. We also thank the 
anonymous referee for his careful reading of the manuscript and for his comments on the work
of Connes, Flato and Sternheimer.
\section{Notation and conventions}
All rings and ring homomorphisms are unital. We mix homological and cohomological indices,
using the convention $X_i=X^{-i}$.

\section{Preliminaries on the Hochschild and cyclic  complexes}
\label{ref-4-3}
In this section remind the reader about  the basic operations on the
Hochschild and cyclic complexes. The reason for putting this section
first is that it also allows us to introduce some notation.
Readers vaguely familiar with the material may safely
skip to the next section.

\subsection{Notation}

Let $R$ be a commutative ring and assume that $B$ is an 
$R$-algebra.  
Let $\C_\bullet(B)$ and $\C^\bullet(B)$ denote the usual relative Hochschild chain and
cochain complexes of $B/R$. Thus
\begin{align*}
\C^\bullet(B)&= \bigoplus_{n}\Hom_R(\Sigma B^{\otimes n},B)\\
\C_\bullet(B)&=  \bigoplus_n B\otimes(\Sigma B)^{\otimes n}
\end{align*}
where here and below, \emph{all unadorned tensor products are over $R$}.
We also use
\begin{align*}
\mathfrak{C}^\bullet(B)&= \Sigma\C^\bullet(B)\label{cohom}\\
&=\bigoplus_{n}\Hom_R(\Sigma B^{\otimes n},\Sigma B)
\end{align*}
We will also consider the \emph{normalized versions} of these objects.
\begin{align*}
\bar{\C}^\bullet(B)&= \bigoplus_{n}\Hom_R(\Sigma (B/R)^{\otimes n},B)\\
\bar{\C}_\bullet(B)&= \bigoplus_n B\otimes \Sigma (B/R)^{\otimes n}
\end{align*}
and a similar definition for $\bar{\mathfrak{C}}^\bullet(B)$.
It is well-known that the obvious maps $\bar{\C}^\bullet(B)\mor \C^\bullet(B)$ and $\C_\bullet(B)\mor \bar{\C}_\bullet(B)$ are quasi-isomorphisms \cite[Thm 8.3.8, Lemma 8.3.7]{weibel}.

If $x\in \mathfrak{C}^n(B)$ then we write
$|x|=n-1$. Thus $|x|$ refers
to the cohomological degree of $x$. 
\subsection{The Hochschild cochain complex}
\label{ref-4.2-4}
The standard algebraic structures on the Hochschild cochain complex
can all be deduced from its structure as a \emph{brace algebra}  \cite{VG}. 
Recall that the braces are maps
\[
\f{C}^\b(B) \tr\ldots \tr \f{C}^\b(B) \mor \f{C}^\b(B):x\otimes x_1\otimes\cdots\otimes
x_m\mapsto x\{x_1,\ldots,x_m\}
\]
defined by
\begin{align*}
&x\{x_1, \ldots, x_m\}(b_1, \ldots b_n)=\\
& \sum_{0 \le i_1 \ldots \le i_m \le n} (-1)^\epsilon x(b_1, \ldots b_{i_1},x_1(b_{i_1+1}, \ldots b_{i_1+\vert x_1 \vert+1}), \ldots, b_{i_m}, x_m(b_{i_m+1}, \ldots ,b_{i_m+\vert x_m \vert+1}), \ldots b_n)
\end{align*}
where $\epsilon= \sum_1^m\vert x_k\vert i_k$ 
The corresponding \emph{Gerstenhaber Lie bracket} on $\mathfrak{C}^\bullet(B)$ is 
\[
[x,y]= x\{y\} -(-1)^{\vert x \vert\vert y \vert}y \{x\}
\]
Let $\mu\in \mathfrak{C}^1(B)=\Hom(\Sigma B\otimes \Sigma B,\Sigma B)$
denote the \emph{``inverse'' multiplication} $\mu(b_1,b_2)=-b_1b_2$. Then $[\mu,\mu]=0$ and hence 
\begin{equation}
\label{ref-4.1-5}
dx=[\mu,x]
\end{equation}
defines a differential of degree one on
$\mathfrak{C}^\bullet(B)$. 

The \emph{cupproduct} on $\mathfrak{C}^\bullet(B)$  
is defined
by 
\[
x\cup y=(-1)^{|x|}\mu\{x,y\}
\]
This is an associative product of degree one on
$\mathfrak{C}^\bullet(B)$, or equivalently an associative product of degree zero
on $\C^\bullet(B)$. One has \cite{VG}
\begin{enumerate}
\item
$(\f{C}^\b(B),d, [\ ,\ ])$ is a DG-Lie algebra
\item
$(\C^\b(B), d,\cup)$ is a DG-algebra, commutative up to homotopy
\item More generally: ($\C^\b(B),d,[\ ,\ ],\cup$) is a so-called DG-``Gerstenhaber algebra''  up to homotopy
\end{enumerate}

\subsection{The Hochschild chain complex}
\label{ref-4.3-6}
When combining the Hochschild cochain complex with the Hochschild chain complex
things becomes much more intricate \cite{CR,TT}.
We will content ourselves by giving formulas for the basic operations and stating some relations
between them. We refer to \cite{TT} for more details.

The first basic operation is the \emph{contraction}.
\[
i_x(b_0\tr \ldots \tr b_n):=b_0x(b_1, \ldots b_d) \tr b_{d+1}\tr \ldots \tr b_n
\]
for $x\in \mathfrak{C}^\bullet(B)$ and $b_0\tr \ldots \tr b_n\in
\C_\bullet(B)$. 
This is an action of $\mathfrak{C}^\bullet(B)$
on $\C_\bullet(B)$ satisfying $|i_x|=|x|+1$ and
\begin{equation}
\label{ref-4.2-7}
i_xi_y=(-1)^{(|x|+1)(|y|+1)}i_{y\cup x}
\end{equation}
The contraction is often written as a \emph{capproduct}: $i_x(-)=x\cap -$.

The second basic operation is the \emph{Lie derivative}
\begin{align*}
L_x(b_0\tr \ldots \tr b_n):&=\sum_{i=0}^{n-\vert x\vert-1} (-1)^{\vert x\vert i}b_0\tr \ldots \tr b_i \tr x(b_{i+1}, \ldots , b_{i+\vert x\vert+1 })\tr \ldots \tr b_n\\
&+\sum_{i=n-\vert x\vert}^n (-1)^{n(i+1)+|x|}x(b_{i+1}, \ldots b_n, b_0, \ldots , b_{\vert x\vert -n+i})\tr \ldots \tr b_i
\end{align*}
 The Lie derivative defines a graded Lie action of
$\mathfrak{C}^\bullet(B)$ on $\C_\bullet(B)$. Explicitly: $|L_x|=|x|$ and
\begin{equation}
\label{ref-4.3-8}
[L_x,L_y]=L_{[x,y]}
\end{equation}
The \emph{Hochschild differential} on $\mathfrak{C}_\bullet(B)$ is defined as $\mathsf{b}=L_\mu$.  From \eqref{ref-4.1-5} and \eqref{ref-4.3-8} one obtains
\begin{equation}
\label{ref-4.4-9}
[\mathsf{b},L_x]=L_{dx}
\end{equation}
Hence
$(\C_\bullet(B),\mathsf{b})$ is a DG-Lie module over $\mathfrak{C}^\bullet(B)$. 
One also has compatibility with the contraction:
\begin{equation}
\label{ref-4.5-10}
[\mathsf{b},i_x]+i_{dx}=0
\end{equation}

The last basic operation we need is the \emph{Connes differential}. 
\[
\mathsf{B}:\C_\bullet(B)\mor \C_\bullet(B)
\]
with formula (see e.g.\ \cite[(2.1.7.3)]{Loday1})
\begin{align*}
\mathsf{B}(b_0\tr \ldots \tr b_n)&= \sum_{i=0}^n (-1)^{ni} 1\tr b_i \tr \ldots \tr b_n \tr b_0 \tr \ldots \tr b_{i-1}\\
\\&\quad +
\sum_0^n (-1)^{n(i+1)} b_{i-1}\tr 1\tr b_i \tr \ldots \tr b_n \tr b_0 \tr \ldots \tr b_{i-2}
\end{align*}
It is  well-known that $|\mathsf{B}|=-1$, 
$\mathsf{b}\mathsf{B}+\mathsf{B}\mathsf{b}=0$, $\mathsf{B}^2=0$.

\medskip

Some of the following identies hold only for normalized chains/cochains. Note that if $x\in \bar{\mathfrak{C}}^\bullet(B)$ then $i_x,L_x$ are well-defined operations
on $\bar{\C}_\bullet(B)$.
\begin{lemma} Assume $x\in \bar{\mathfrak{C}}^\bullet(B)$. Then on $\C_\bullet(B)$ we have
\begin{equation}
\label{ref-4.6-11}
[\mathsf{B},L_x]=0
\end{equation}
\end{lemma}
The formula \eqref{ref-4.6-11} does not hold for unnormalized cochains. 
\subsection{The negative cyclic complex}
Let $u$ be a variable of degree two and put
\[
\overline{\CC}^{\,-}_\bullet(B)=\bar{\C}_\bullet(B)[[u]]
\]
Equipped with the \emph{cyclic differential} $\mathsf{b}+u\mathsf{B}$, this is the
 \emph{normalized negative cyclic complex}. In the sequel operations on $\bar{{\C}}_\bullet(B)$
will be (tacitly) extended to $\overline{\CC}^{\,-}_\bullet(B)$ by making
them $u$-linear.  This applies in particular to $i_x$ and $L_x$.
Combining \eqref{ref-4.4-9}\eqref{ref-4.6-11} we obtain
\begin{equation}
\label{ref-4.7-12}
[\mathsf{b}+u\mathsf{B},L_x]=L_{dx}
\end{equation}
The compatibility of $i_x$ with the cyclic differential is more subtle. In \cite{TT} 
(see also \cite{Getzler1}) Tamarkin and Tsygan define for $x\in \bar{\mathfrak{C}}^\bullet(B)$ 
a graded endomorphism $S_x$ of $\bar{\C}_\bullet(B)$  (depending linearly on $x$)
such that  $|S_x|=|x|-1$ and such that the following identity holds
\begin{equation}
\label{ref-4.8-13}
[\mathsf{b}+u\mathsf{B}, i_x+uS_x]+i_{dx}+uS_{dx}=uL_{x}
\end{equation}
on $\overline{\CC}^{\,-}_\bullet(B)$ (they attribute this result to Rinehart \cite{Rinehart}).
This identity will be important for us in the sequel. Note that it implies \eqref{ref-4.7-12}.

The following is a special case of \cite[Prop.\ 3.3.4]{TT}. 
\begin{lemma} \label{ref-4.2-14}
Let $x,y\in\bar{\mathfrak{C}}^\bullet(B)$ be such that $dx=dy=0$. Then $[L_x,i_y+uS_y]$ is homotopic
to $(-1)^{|x|}(i_{[x,y]}+uS_{[x,y]})$.
\end{lemma}
\subsection{A comment on base change}
\label{ref-4.5-15}
If $A$ is a $k$-algebra and $R$ is a commutative $k$-algebra then for
$B=A\otimes_k R$ it is clear that
$\C_{R,\bullet}(B)\cong\C_\bullet(A)\otimes_k R$ (where contrary to our usual conventions we have now made the base ring explicit in the notation). Since the negative
cyclic complex involves a product this is not true for
$\CC^-_{R,\bullet}(B)$. However it is true if $R$ is finite
dimensional. Similarly in that case we have $\C^\bullet_R(B)\cong \C^\bullet(A)\otimes_k R$.
In the sequel we will not mention these base change isomorphisms explicitly.
\subsection{Some comments on signs}
\label{ref-4.6-16}
In the previous sections the operations
$i_x,L_x,S_x,\mathsf{b},\mathsf{B}$ of degree $|x|+1$, $|x|$, $|x|-1$, $1$, $-1$
 were defined as acting on
$\bar{\C}_\bullet(B)$.
We define corresponding operations on shifts $\Sigma^r
\bar{\C}_\bullet(B)$ in the usual way.
\begin{align*}
  i_x(s^r b)&=(-1)^{r(|x|+1)}s^r i_x(b)\\
  L_x(s^r b)&=(-1)^{r|x|}s^r L_x(b)\\
  S_x(s^r b)&=(-1)^{r(|x|-1)}s^r S_x(b)\\
  \mathsf{b}(s^rb)&=(-1)^r s^r \mathsf{b}(b)\\
  \mathsf{B}(s^rb)&=(-1)^r s^r \mathsf{b}(B)
\end{align*}
where $s$ is the \emph{degree change operator} $|sb|=|b|-1$.

The relations between $i_x,L_x,S_x,\mathsf{b},\mathsf{B}$ stated in \S\ref{ref-4.3-6},\S\ref{ref-4.3-6}
carry over to all shifts $\Sigma^r \bar{\C}_\bullet(B)$ without
any sign changes, since all terms in the identities (necessarily) have
the same degree.

\section{Preliminaries on Calabi-Yau algebras}
In this section we extend Ginzburg's definition of  Calabi-Yau algebras to
the relative case. 

Let $R$ be a commutative ring. For an $R$-algebra $B$
we put $B^e=B\otimes_R B^\circ$.  We use without further comment
 the standard equivalences between the categories of left $B^e$-modules, right
$B^e$-modules and $B$-bimodules which are $R$-central.

A $B^e$-module is called \emph{perfect} if it has a finite resolution
by finitely generated projective $B^e$-modules. If $B$ is $R$-flat and
$B$ is a perfect $B^e$-module then we say that $B$ is
\emph{homologically smooth} over $R$. The implicit flatness hypothesis ensures
that $B^e=B\otimes_R B^\circ$ is the correct definition from a derived
point of view.  We could have avoided this hypothesis
by first replacing $B$ by an $R$-flat DG-resolution but for simplicity we
have chosen not to do this.

\begin{definition} (Ginzburg \cite{Gi}) \label{ref-5.1-17}
An $R$-Calabi-Yau algebra of
dimension~$d$ is a pair $(B,\eta)$ where 
\begin{enumerate}
\item
$B$ is an $R$-algebra $B$ which is homologically smooth over $R$;
\item
$\eta$ is an isomorphism $\RHom_{B^e}(B,B^e) \mor \Sigma^{-d} B$  in $D(B^e)$. 
\end{enumerate}
\end{definition}
\begin{remark} Note that the amount of freedom for $\eta$ is quite
limited. If $(B,\eta)$, $(B,\eta')$ are Calabi-Yau then there exists $z\in Z(B)$
such that $\eta'=z\eta$ (see \cite{davison}).
\end{remark}

Recall that if $M$ is a complex of  $B^e$-module then its \emph{Hochschild homology and cohomology}
are respectively defined as
\begin{align*}
\HH_i(B,M)&=H^{-i}(M\otimes^L_{B^e} B)\\
\HH^i(B,M)&=H^i(\operatorname{RHom}_{B^e}(B,M))
\end{align*}
As usual we write $\HH_i(B)=\HH_i(B,B)$ and similarly $\HH^i(B)=\HH^i(B,B)$. 
One has
\[
\HH_i(B)=H^{-i}(\C_{\bullet}(B))
\]
and if $B$ is a projective $R$-module then 
\[
\HH^i(B)=H^{i}(\C^{\bullet}(B))
\]
The operations $[\ ,\ ],\cup,\cap,L,\mathsf{B}$ introduced in \S\ref{ref-4-3} descend
to homology and make the pair $(\HH^\bullet(B),\HH_\bullet(B))$ into a so-called
\emph{calculus} structure \cite{TT}. Up to suitable, and for us irrelevant, signs
$\cup$ is the Yoneda products on $\HH^\bullet(B)=\Ext^\bullet_{B^e}(B,B)$ and $\cap$ is the action of
$\HH^\bullet(B)$ on $\HH_\bullet(B)=H^{-\bullet}(B\otimes^L_{B^e} B)$ through its action
on the second factor (see e.g.\ \cite[Prop 11.1, 12.1]{CRVdB}). 
\begin{lemma}\label{ref-5.3-18}
  Let $B$ be a homologically smooth algebra. Then for $M$ a perfect
  $B^e$-module there is a canonical isomorphism
\begin{equation}
\label{ref-5.1-19}
\HH_i(B,M) \cong \Hom_{B^e}(\Sigma^i \RHom_{B^e}(M,B^e),B)
\end{equation}
in $D(R)$. 
\end{lemma}

\begin{proof}
Since $M$ is perfect we may replace it with a complex of finitely
generated projective $B^e$-modules. In this way we reduce to $M=B^e$ which
is an easy verification.
\end{proof}

\begin{definition}
Let $B$ be a homologically smooth algebra $R$ and $\eta \in \HH_d(B)$. We say that $\eta$ is \emph{nondegenerate} if its image under \eqref{ref-5.1-19} is an isomorphism.
\end{definition}
This allows us to redefine a $d$-Calabi-Yau algebra over $R$ as a couple $(B,\eta)$ where
$B$ is a homologically smooth $R$-algebra
and $\eta$ is a non-degenerate element of  $\HH_d(B)$. 
Below we will massage this new definition further. Recall the following
\begin{proposition} (``Poincare duality'') \label{ref-5.5-20}
Assume that $(B,\eta)$ is a $d$-Calabi-Yau $R$-algebra. Then for each $i$, the map
\begin{equation}
\label{ref-5.2-21}
\HH^i(B) \mor \HH_{d-i}(B): \mu \fun \mu\cap \eta
\end{equation}
is an isomorphism
\end{proposition}
\begin{proof} The existence of the isomorphism was first stated in \cite{VdB29} without the explicit formula
\eqref{ref-5.2-21}. The formula \eqref{ref-5.2-21} is folklore. For completeness
we include a possible proof. 

The proof of Lemma \ref{ref-5.3-18} shows that there is a canonical isomorphism in $D(R)$
\begin{equation}
\label{ref-5.3-22}
\RHom_{B^e}(\RHom_{B^e}(B,B^e),\overset{\downarrow}{B})\cong B\otimes^L_{B^e}\overset{\downarrow}{B}
\end{equation}
which is compatible with the $\RHom_{B^e}(B,B)$-actions on the marked copies of~$B$.

By definition $\eta\in H^{-d}(B\otimes^L_{B^e}B)$ corresponds under \eqref{ref-5.3-22}
to an isomorphism $\eta^+:\RHom_{B^e}(B,B^e)\rightarrow
\Sigma^{-d}B$. This yields an isomorphism
\begin{equation}
\label{ref-5.4-23}
\RHom_{B^e}(\Sigma^{-d}B,\overset{\downarrow}{B})\mor \RHom_{B^e}(\RHom_{B^e}(B,B^e),\overset{\downarrow}{B})
:\theta\mapsto \theta\circ \eta^+
\end{equation}
also compatible with the  marked $\RHom_{B^e}(B,B)$-actions. Composing \eqref{ref-5.3-22}\eqref{ref-5.4-23}
gives us an isomorphism 
\[
\xi:\RHom_{B^e}(\Sigma^{-d}B,\overset{\downarrow}{B})\mor  B\otimes^L_{B^e}\overset{\downarrow}{B}
\]
which sends $\operatorname{Id}_B$ to $\eta$.

According to the discussion preceding Lemma \ref{ref-5.3-18}, the compatibility with
the $\RHom_{B^e}(B,B)$-actions  implies that $\xi$ transforms
$\cup$ into $\cap$ on the level of cohomology. More precisely
\[
\xi(\mu\cup \sigma)=\pm \mu\cap \xi(\sigma)
\]
The lemma now follows by taking $\sigma=\operatorname{Id}_B$.
\end{proof}
\begin{corollary}
\label{ref-5.6-24}
Assume that  $(B,\eta)$ is a $d$-Calabi-Yau $R$-algebra. Then
\begin{align*}
\HH^i(B)=0&\qquad\text{for $i\not \in [0,d]$}\\
\HH_i(B)=0&\qquad\text{for $i\not \in [0,d]$}\qed
\end{align*}
\def\qed{}\end{corollary}
As before let $\CC^-_\bullet(B)=(\C_\bullet(B)[[u]],\mathsf{b}+u\mathsf{B})$ be the negative cyclic complex and
denote its corresponding homology by $\HC_\bullet^-(B)$.
\begin{proposition}\label{ref-5.7-25}
  Let $(B,\eta)$ be a $d$-Calabi-Yau $R$-algebra. Then 
$\HC^-_i(B)=0$ for $i>d$ and furthermore
the map 
\[
\pi:\CC^-_\b(B)
  \mor \C_\b(B): \sum b_iu^i \fun b_0
\] induces an isomorphism
$
\HC^-_d(B) \cong \HH_d(B)
.$
\end{proposition}
\begin{proof}
  We use a spectral sequence argument.  We view $\CC^-_\bullet(B)$ as a double complex
with $\mathsf{b}$ pointing vertically upwards and $u\mathsf{B}$ pointing horizontally to the
right. 
By Corollary \ref{ref-5.6-24} we have
  $\HH_i(B)=0$ for $i>d$. 
Hence if we filter $\CC^-_\bullet(B)$
  by column degree then the $E^1$ term of the resulting spectral sequence looks like
\[
\xymatrix{
0& \HH_{d-2}(B)\ar[r]^-{u\mathsf{B}} & u\HH_{d-1}(B)\ar[r]^-{u\mathsf{B}} &u^2\HH_d(B) \\
0& \HH_{d-1}(B)\ar[r]^-{u\mathsf{B}} & u\HH_d(B) & 0\\
0&\HH_d(B)& 0 & 0\\
0& 0 & 0 & 0
}
\]
From this the
  result follows.
\end{proof}
\begin{definition}
  Let $B$ be a homologically smooth $R$-algebra. We say that
  an element $\eta \in \HC^-_d(B)$ is non-degenerate if $\pi(\eta)$ is
  non-degenerate.
\end{definition}
The leads to the following redefinition of a Ginzburg $d$-Calabi-Yau $R$-algebra which
we use below.
\begin{definition}  (Restatement of Definition \ref{ref-5.1-17}.)
\label{ref-5.9-26}
A \emph{Calabi-Yau algebra of dimension $d$ over $R$} is a 
couple $(B,\eta)$ where $B$ is a homologically smooth $R$-algebra
and $\eta$ is a non-degenerate element of $\HC_d^-(B)$.
\end{definition}
We have shown that this definition is equivalent to Ginzburg's original definition. 
In the more general setting of  DG-algebras this is no longer the case. It is generally
believed that Definition \ref{ref-5.9-26} is the ``correct'' definition for
a $d$-Calabi-Yau algebra in the DG-case. This is the point of view of Kontsevich-Soibelman in \cite{KS2} and also of Keller \cite{Kl1}.
\section{Deformations of Calabi-Yau algebras}
\label{ref-6-27}
In this section we fix a $d$-Calabi-Yau $k$-algebra $(A,\eta_0)$  as in
Definition \ref{ref-5.9-26}. We will study the deformations of $A$ as a Calabi-Yau algebra.

\medskip

Let $\Nilp$ be the category of commutative, finite
dimensional, local $k$-algebras $(R,m)$ such that $R/m=k$. 
For $(R,m)\in
\Nilp$ we define a groupoid $\Def_{A,\eta_0}(R)$ as follows: the
objects in $\Def_{A,\eta_0}(R)$ are triples $(B,s,\eta)$ such that $B$ 
is an $R$-flat $R$-algebra,
$s:B\rightarrow A$ is an $R$-algebra map inducing an
isomorphism $B\otimes_R k\rightarrow A$ and  $\eta$ is an element of
$\HC^-_d(B)$ such that $s(\eta)=\eta_0$.

A morphism $(B_1,s_1,\eta_1)\rightarrow (B_2,s_2,\eta_2)$ is a commutative
diagram
 \begin{displaymath}
 \xymatrix{
 B_1\ar[rr]^\phi\ar[dr]_{s_1} & &B_2\ar[dl]^{s_2}\\
& A
 } \end{displaymath}
such $\eta_2=\phi(\eta_1)$. One sees that $\Def_{A,\eta_0}$ becomes a pseudo-functor
$\Nilp\rightarrow \Gd$ in the obvious way. 

To be able to rightfully claim that $\Def_{A,\eta_0}$ describes the Calabi-Yau deformations
of $(A,\eta_0)$ we need the following elementary lemma. 
\begin{lemma} Assume that $(B,s,\eta)\in \Def_{A,\eta_0}(R)$. Then $(B,\eta)$ is $d$-Calabi-Yau.
\end{lemma}
\begin{proof} We have to show that $B$ is a perfect $B^e$-module and that $\eta$ induces
an isomorphism $\eta^+:\RHom_{B^e}(B,B^e)\rightarrow \Sigma^{-d} B$. 

Since $R$ is finite dimensional every flat $R$-module is
$R$-projective. This applies in particular to $B$ and $B^e$.  Let 
\[
0\rightarrow P_u\rightarrow\cdots\rightarrow P_0\rightarrow A\rightarrow 0
\]
be a finite resolution of $A$ by finitely generated projective $A^e$-modules. It is easy
to see that this resolution can be lifted step by step to a resolution
\[
0\rightarrow Q_u\rightarrow\cdots\rightarrow Q_0\rightarrow B\rightarrow 0
\]
where the $Q_i$ are finitely generated projective $B^e$-modules satisfying $Q_i\otimes_R k\cong P_i$. In particular $B$ is perfect. 

It also follows that $H=\operatorname{cone}\eta^+$ is perfect. 
It it easy to that $\eta^+\otimes^L k\cong \eta_0^+$ and hence
$(\operatorname{cone}\eta^+)\otimes^L k\cong \operatorname{cone}(\eta^+\otimes^L k)\cong 
\operatorname{cone}\eta_0^+=0$. If now
suffices to note that if $H$ is perfect and $H\otimes^L k=0$ then $H=0$.
\end{proof}
We now introduce a variant of the groupoid $\Def_{A,\eta_0}(R)$ which is easier to describe
 cohomologically. We remind the reader of the base change convention exhibited in \S\ref{ref-4.5-15} which we will use throughout.
As in \S\ref{ref-4.2-4} let $-\mu_0\in \mathfrak{C}^1(A)$ be the 
multiplication map  on $A$ and let $\hat{\eta}_0$ be a lift of $\eta_0$ to $\overline{\CC}^{\,-}_d(A)$. We
define an associated groupoid $\Def^\flat_{A,\hat{\eta}_0}(R)$ as follows. The objects
are couples $(\mu,\eta)$ where
\begin{enumerate}
\item $\mu\in \mathfrak{C}^1(A)\otimes_k R$ is such that $-\mu$ defines a unital associative multiplication
on $A\otimes_k R$;
\item $\mu\mod m=\mu_0$;
\item $\eta\in \overline{\CC}^{\,-}_d(A)\otimes_k R$;
\item $(L_\mu+u\mathsf{B})(\eta)=0$;
\item $\eta\mod m=\hat{\eta}_0$.
\end{enumerate}
For (4)\ recall that $L_{\mu}+u\mathsf{B}$ is the cyclic differential for
the algebra $(A\otimes_k R, \mu)$.
A morphism $(\mu_1,\eta_1)\rightarrow (\mu_2,\eta_2)$ in $\Def^\flat_{A,\mu_0,\hat{\eta}_0}(R)$ is a couple 
$(\phi,\xi)$ where 
\begin{enumerate}
\item
$\phi$ is a unital map of $R$-algebras $\phi: (A\otimes_k R, -\mu_1) \mor (A\otimes_k R, -\mu_2)$;
\item $\phi$ is the identity modulo $m$;
\item $\xi$ is an element of $\overline{\CC}^{\,-}_{d+1}(A)\otimes_k m$;
\item \label{ref-4-28}
$
(L_{\mu_2}+u\mathsf{B})(\xi)=\phi(\eta_1)-\eta_2
$.
\end{enumerate}

The composition of morphisms
\[
(\mu_1,\eta_1)\xrightarrow{(\phi',\xi')} (\mu_2,\eta_2)
\xrightarrow{(\phi,\xi)}(\mu_2,\eta_2)
\]
is defined by
\begin{equation}
\label{ref-6.1-29}
(\phi,\xi) \circ
(\phi',\xi')= (\phi \circ \phi',\phi(\xi')+\xi)
\end{equation}
Below we will often use the notation $\bar{\eta}$ for the cohomology
class of a cocycle $\eta$.
\begin{proposition}
\label{ref-6.2-30}
The morphism of groupoids
\begin{align*}
  \operatorname{Ob}(\Def^\flat_{A,\hat{\eta}_0}(R))\mor 
\operatorname{Ob}(\Def_{A,\eta_0}(R))&:(\mu,\eta)\mapsto ((A\otimes_k R,-\mu),``\operatorname{mod} m",\bar{\eta})\\
  \operatorname{Mor}(\Def^\flat_{A,\hat{\eta}_0}(R))\mor
  \operatorname{Mor}(\Def_{A,\eta_0}(R))&:(\phi,\xi)\mapsto \phi
\end{align*}
is essentially surjective on objects and surjective on morphisms. 
\end{proposition}
\begin{proof} We first prove essential surjectivity. Let $(B,s,\psi)
\in \Def_{A,\eta_0}(R)$. Then since $R$ is finite dimensional local and $B$ is $R$-flat we have an isomorphism
$B\cong A\otimes_k R$ as $R$-modules and it is easy to see that
this isomorphism may be chosen to make the following diagram commutative
\[
\xymatrix{
B\ar[rr]^\phi\ar[dr]_s && A\otimes_k R \ar[dl]^{\mod m}\\
& A
}
\]
We now transfer the multiplication on $B$ to $A\otimes_k R$ where it
becomes an element of $-\mu\in\mathfrak{C}^1(A)\otimes_k R$ which modulo $m$ is
equal to $-\mu_0$. We do the same with $\psi\in \HC^-_d(B)$
and we choose an element $\eta\in \overline{\CC}^{\,-}_d(A)\otimes_k
R$ such that $(L_\mu+u\mathsf{B})(\eta)=0$, $\bar{\eta}=\phi(\psi)$.
Thus in $\Def_{A,\eta_0}(R)$ we have
\[
(B,s,\psi)\cong ((A\otimes_k R,-\mu),-\operatorname{mod} m,\bar{\eta})
\]
This proves essential surjectivity.
Now we prove surjectivity on morphisms. Let $(\mu_1,\eta_1)$, $(\mu_2,\eta_2)\in
\operatorname{Ob}(\Def^\flat_{A,\hat{\eta}_0}(R))$ and let $\phi$ be a unital algebra morphism
\[
(A\otimes_k R,-\mu_1)\mor (A\otimes_k R,-\mu_2)
\]
inducing the identity modulo $m$ and satisfying $\phi(\bar{\eta}_1)=
\bar{\eta}_2$.

It follows that $\phi(\eta_1)-\eta_2$ is a boundary in the negative
cyclic complex of $(A\otimes_k R,-\mu_2)$. In other words there exists
$\xi\in \overline{\CC}_{d+1}^{\,-}(A)\otimes_k R$ such that 
\[
\phi(\eta_1)-\eta_2=(L_{\mu_2}+u\mathsf{B})(\xi)
\]
We have to show that we may choose $\xi\in \overline{\CC}_{d+1}^{\,-}(A)\otimes_k m$

Since $\phi$ is the identity modulo $m$ we have
\begin{align*}
\phi(\eta_1)-
\eta_2 \mod m&=\eta_1-\eta_2\mod m\\
&=\hat{\eta}_0-\hat{\eta}_0\\
&=0
\end{align*}
It follows that $\phi(\eta_1)-\eta_2\in \overline{\CC}^{\,-}_d(A)\otimes_k m$ 
 and hence $d\xi\mod m=0$. Since $\HC^-_{d+1}(A)=0$ by Proposition \ref{ref-5.7-25} we see that there exists
$\gamma \in \overline{\CC}^{\,-}_{d+2}(A)\otimes_k R$ such that $(L_{\mu_2}+u\mathsf{B})(\gamma)\cong \xi\mod m$.  In other words
\[
\xi'=\xi-(L_{\mu_2}+u\mathsf{B})(\gamma)\in \overline{\CC}^{\,-}_d(A)\otimes_k m
\]
Then the couple $(\phi,\xi')$ is a pre-image for $\phi$.
\end{proof}
For completeness we state the following.
\begin{proposition}
\label{ref-6.3-31}
Let $\hat{\eta}_0'\in \overline{\CC}^{\,-}_0(A)$ be a different lift of $\eta_0$. Then
$\Def^\flat_{A,\hat{\eta}'_0}(R)$ and $\Def^\flat_{A,\hat{\eta}_0}(R)$ are isomorphic.
\end{proposition}
We could easily prove this here directly, however we will postpone the
proof till \S\ref{ref-8-36} where we reinterprete
$\Def^\flat_{A,\hat{\eta}_0}(R)$ in terms of the Maurer-Cartan equation.

\section{The Maurer-Cartan formalism}
\label{ref-7-32}
In this section we briefly recall
the construction of the deformation functor associated to a DG-Lie algebra.

\medskip

Let $\f{h}^\b$ be a DG-Lie algebra over $k$.  The set 
\[
\MC(\mathfrak{h}^\bullet)\overset{\text{def}}{=}\left\{y \in \f{h}^1 \, \left\vert
dy +\frac{1}{2}[y,y]=0 \right.\right\}
\]
is the set of solutions to the \emph{Maurer-Cartan equation} in
$\f{h}^\b$. It has a natural structure of a groupoid which we now describe.

Assume that $\mathfrak{n}$ is a nilpotent Lie algebra and let
$\widehat{U}(\mathfrak{n})$ be the enveloping algebra of $\mathfrak{n}$, completed
at the augmentation ideal. Then the group $\exp(\mathfrak{n})$ is by definition
the set of group like elements in $\widehat{U}(\mathfrak{n})$.
 It is well-known and easy to see that there is a bijection
\[
\exp:\mathfrak{n}\mor \exp(\mathfrak{n}):n\mapsto e^n
\]
between the primitive and the group like elements in $\widehat{U}(\mathfrak{n})$.

Now assume that $\f{h}^\b$ is nilpotent. Then
$\widehat{U}(\mathfrak{h}^0)$ acts on the graded Lie algebra~$\mathfrak{h}^\bullet$ using the adjoint action and hence so does the
\emph{gauge group}
$G(\mathfrak{h}^\bullet)\overset{\text{def}}{=}\exp(\mathfrak{h}^0)$. This
action does not commute with the differential and in particular it does not preserve $\MC(\mathfrak{h}^\bullet)$. However the following
modified \emph{gauge action} does:
\begin{equation}
\label{ref-7.1-33}
\begin{aligned}
\exp(x)\ast y&\overset{\text{def}}{=}e^{\ad x}(y)-\frac{e^{\Ad x}-1}{\Ad x}(dx)\\
&=e^{\Ad x}(y)-\sum_{n=0}^\infty \frac{1}{(n+1)!} (\Ad x)^n (dx)
\end{aligned}
\end{equation}
where $x\in\mathfrak{h}^0$, $y\in \mathfrak{h}^1$ and
$(\ad x)(u)=[x,u]$.

An elegant derivation of this action is given by Manetti~\cite[\S
V.4]{Manetti}. One first formally adjoins an element $\delta$ of
degree one to $\mathfrak{h}^\bullet$ such that $dx=[\delta,x]$,
$d\delta=0$ and $[\delta,\delta]=0$. Then \eqref{ref-7.1-33} can be
rewritten as:
\begin{equation}
\label{ref-7.2-34}
\exp(x)\ast y=e^{\ad x}(y+\delta)-\delta
\end{equation}
This action preserves $\MC(\mathfrak{h}^\bullet)$ since for $y\in\mathfrak{h}^1$:
\[
y\in\MC(\mathfrak{h}^\bullet)\iff [y+\delta,y+\delta]=0
\]
In the sequel we view $\MC(\mathfrak{h}^\bullet)$ as a groupoid through
the $G(\mathfrak{h}^\bullet)$-action.

\medskip

If $y\in \MC(\mathfrak{h}^\bullet)$ then by definition $\mathfrak{h}^\bullet_y$ is
the DG-Lie algebra which is $\mathfrak{h}^\bullet$ as graded Lie algebra but which
has the deformed differential $d_y=d+[y,-]$. Using \eqref{ref-7.2-34} one easily
shows that for $x\in \mathfrak{h}^0$.
\begin{equation}
\label{ref-7.3-35}
e^{\ad x}:\mathfrak{h}^\bullet_y\mor \mathfrak{h}^\bullet_{\exp(x)\ast y}
\end{equation}
is an isomorphism of DG-Lie algebras 

\medskip

Assume $(R,m)\in\Nilp$. Given an arbitrary DG-Lie algebra 
$\mathfrak{g}^\bullet$ over $k$, the vector space $\f{g}^\b \tr_k m$ becomes a
nilpotent DG-Lie algebra. We define $\MCcal(R)$ as $\MC(\mathfrak{g}^\bullet\otimes_k m)$ equipped with the groupoid structure introduced above. 
In this way we obtain a pseudo-functor $\MCcal:\Nilp\mor \Gd$. This
is the ``deformation functor'' associated to $\mathfrak{g}^\bullet$.

\section{The DG-Lie algebra {\mathversion{bold} $\mathfrak{D}^\bullet(A,\eta)$}}
\label{ref-8-36}
Below $(A,\bar{\eta}_0)$ is a $d$-Calabi-Yau $k$-algebra where
$\eta_0\in \overline{\CC}_d^-(A)$ satisfies $(L_{\mu_0}+u\mathsf{B})(\eta_0)=0$, with
$-\mu_0\in\mathfrak{C}^1(A)$ being the multiplication on $A$. 
In this section we associate a DG-Lie algebra  $\mathfrak{D}^\bullet(A,\eta_0)$ to $A$
and prove that its deformation functor (see \S\ref{ref-7-32}) is isomorphic to the 
functor $\Def^\flat_{A,\eta_0}$ introduced
in~\S\ref{ref-6-27}.

\medskip

If $\f{g}^\b$ is a DG-Lie algebra and $M^\b$ a $\f{g}^\b$-module then the direct sum complex $\f{g}^\b \ds M^\b$ becomes a DG-Lie algebra when endowed with the following bracket:
\[
[(g,m),(g',m')]:=([g,g'], gm'-(-1)^{\vert g'\vert \vert m\vert} g'm)
\]
The resulting DG-Lie algebra is called the \emph{semi-direct product} 
of $\mathfrak{g}^\bullet$ and $M^\bullet$ and is denoted by $\f{g}^\b \ltimes M^\b$

By \eqref{ref-4.3-8}\eqref{ref-4.7-12} (see also \S\ref{ref-4.6-16}) we have a DG-Lie action 
\[
\bar{\f{C}}^\b(A)\times \Sigma^{-d-1}\overline{\CC}_\b^-(A)\mor \Sigma^{-d-1}\overline{\CC}_\b^-(A):(x,\eta)\mapsto L_x\eta
\]
and  we can form the corresponding semi-direct product
$\mathfrak{D}^\bullet(A)^\sharp=\bar{\f{C}}^\b(A) \ltimes \Sigma^{-d-1}\overline{\CC}_\b ^-(A)$.

The element $x=(0,s^{-d-1}\eta_0)\in \mathfrak{D}^\bullet(A)^\sharp$ satisfies $dx=0$
and $[x,x]=0$. So it satisfies the Maurer-Cartan equation. Put
$\mathfrak{D}^\bullet(A,\eta_0)=\mathfrak{D}^\bullet(A)_x^\sharp$, with notation
as in \S\ref{ref-7-32}.

\begin{theorem}\label{ref-8.1-37}
Let $(R,m)\in \Nilp$. There is an isomorphism of groupoids
\[
\Phi(R):\MC(\f{D}^\b(A,\eta_0)\otimes_k m)  \mor \Def^\flat_{A,\eta_0}(R)
\]
which on objects is given by
\begin{equation}
\label{ref-8.1-38}
(\mu,s^{-d-1}\eta)\mapsto(\mu_0+\mu,\eta_0+\eta)
\end{equation}
\end{theorem}
\begin{corollary}\label{ref-8.2-39}
There is a natural transformation of pseudo-functors
\[
\Phi:\mathcal{MC}(\f{D}^\b(A,\eta_0)) \mor \Def^\flat_{A,\eta_0}
\]
which, when evaluated on $R\in \Nilp$, is an isomorphism of groupoids.
\end{corollary}

We shall prove Theorem \ref{ref-8.1-37} by combining some lemmas. Throughout we fix $(R,m)\in \Nilp$. The following lemma says that $\Phi(R)$ behaves correctly on objects.
\begin{lemma}\label{ref-8.3-40}
Let $\mu \in \bar{\mathfrak{C}}^\bullet_1(A)\otimes_k m$ and $\eta \in \overline{\CC}^{\,-}_d(A)\otimes_k m$. The following are equivalent:
\begin{enumerate}
\item
$(\mu,s^{-d-1}\eta) \in \MC(\f{D}^\b(A,\eta_0)\otimes_k m)$;
\item
$(\mu_0+\mu,\eta_0+\eta)\in \Def^\flat_{A,\eta_0}(R)$.
\end{enumerate}
\end{lemma}

\begin{proof}
We will work out what it means for $(\mu,s^{-d-1}\eta) \in \f{D}^1(A,\eta_0)\otimes_k m$
to satisfy the Maurer-Cartan equation. To simplify the notations we write
$\eta'_0=s^{-d-1}\eta_0$, $\eta'=s^{-d-1}\eta$. We compute
\begin{align*}
\frac{1}{2}[(\mu,\eta'),(\mu,\eta')]+d_{\f{D}}(\mu,\eta')
&= \frac{1}{2}([\mu,\mu], 2L_\mu(\eta'))+([\mu_0,\mu], (L_{\mu_0}+u\mathsf{B})(\eta'))+[(0,\eta_0'), (\mu,\eta')]\\
&=\frac{1}{2}([\mu,\mu], 2L_\mu(\eta'))+([\mu_0,\mu], (L_{\mu_0}+u\mathsf{B})(\eta'))+
[0,L_{\mu}(\eta_0')]\\
&= (\frac{1}{2}[\mu,\mu]+[\mu_0,\mu],L_{\mu}(\eta')+(L_{\mu_0}+u\mathsf{B})(\eta')+ L_{\mu}(\eta_0'))\\
&=([\mu_0+\mu,\mu_0+\mu], (L_{\mu+\mu_0}+u\mathsf{B})(\eta'+\eta'_0))
\end{align*}
where in the last line we have used $[\mu_0,\mu_0]=0$,
$(L_{\mu_0}+u\mathsf{B})(\eta'_0)=0$. Thus if 
$(\mu_0+\mu,\eta_0+\eta)\in \Def^\flat_{A,\eta_0}(R)$ then 
$(\mu,s^{-d-1}\eta)\in  \MC(\f{D}^\b(A,\eta_0)\otimes_k m)$. To prove the
converse the only thing we still need to check is that $-(\mu_0+\mu)$ defines
a \emph{unital} multiplication on $A\otimes_k R$. This follows immediately
from the fact that $-\mu_0$ is unital and $\mu$ is normalized. 
\end{proof}
The next two lemmas will help us describing the gauge group action of
$G(\mathfrak{D}^\bullet(A,\eta))$.
\begin{lemma}\label{ref-8.4-41}
Let $\f{n}$ be a nilpotent Lie algebra over $k$ and let $M$ a  representation of $\f{n}$.
Then there is an isomorphism of groups
\[
\exp(\frak{n})\ltimes M\mor \exp(\frak{n}\ltimes M):
(\exp(n),m)\mapsto  \exp(n,0)\exp(0,m)
\]
\end{lemma}

\begin{proof}
  This is a straightforward verification from the definition of ``$\exp$''
  in \S\ref{ref-7-32} using the fact that
\[
U(\mathfrak{h}\ltimes M)\cong U(\mathfrak{h})\ltimes
\operatorname{Sym}(M)\qed
\]
\def\qed{}\end{proof}
 
\begin{lemma}\label{ref-8.5-42}
Let $\f{g}^\b$ be a nilpotent DG-Lie algebra over $k$ and $M^\b$ a nilpotent DG-module.
Consider the DG-algebra $\f{h}^\b$ which is $\f{g}^\b \ltimes M^\bullet$ as graded
Lie algebras and which is equipped with a deformed differential $(d_{\mathfrak{g}},d_M)+
d_0$ where $d_0: \f{g}^\b \mor M$ is of the form $g\mapsto (-1)^{|g|}gm_0$ for suitable
$m_0\in M^1$.
Then for $g\in\mathfrak{g}^0$, $m\in M^0$ and $(g_1, m_1)\in\mathfrak{h}^1$ we have
\begin{align*}
\exp(g,0)\ast(g_1,m_1)&=(\exp(g)\ast g_1,e^g(m_1-m_0)+m_0)\\
\exp(0,m)\ast(g_1,m_1)&=(g_1,m_1-(g_1+d_M)m)
\end{align*}
\end{lemma}
\begin{proof}
We compute
\begin{align*}
\exp(g,0)*(g_1,m_1)&= e^{\Ad(g,0)}(g_1, m_1) -\sum_n \frac{1}{(n+1)!} \ad^n(g,0)(d_\f{h}(g,0))\\
&=(e^{\ad g}g_1, e^gm_1)-\sum_n \frac{1}{(n+1)!} \ad^n(g,0)(d_\f{g}g,d_0g)\\
&=(e^{\ad g}g_1,e^gm_1) -\sum_n \frac{1}{(n+1)!}(\ad^n(g)(d_\f{g}g),g^{n+1}m_0)\\
&=(e^g*g_1,e^g(m_1-m_0)+m_0)
\end{align*}
Similarly:
\begin{align*}
\exp(0, m)*(g_1, m_1)&=e^{\ad(0, m)}(g_1, m_1) -\sum \frac{1}{(n+1)!} \ad^n(0, m)(d_\f{h}(0,m))\\
&=(g_1,m_1)-(0,g_1m)-(0,d_Mm)\\
&=(g_1, m_1-(g_1+d_M)m)\qed
\end{align*}
\def\qed{}\end{proof}

We will also use the following variant of \eqref{ref-7.2-34}
\begin{lemma}\label{ref-8.6-43}
Let $\f{h}^\b$ be a nilpotent DG-Lie algebra with inner differential $d=[\mu_0,-]$. 
Then for $x\in \mathfrak{h}^0$, $y\in\mathfrak{h}^1$ one has
\[
\exp(x)*y= e^{\ad x}(y+\mu_0)-\mu_0
\]
\end{lemma}

\begin{proof} Direct evaluation of the righthand side yields
the formula \eqref{ref-7.1-33} for $\exp(x)\ast y$.
\end{proof}

\begin{proof}[Proof of Theorem \ref{ref-8.1-37}]
We start by verifying that \eqref{ref-8.1-38} yields indeed a map of groupoids. 
To this end we have to define $\Phi(R)$ on maps. 
Note that by lemma \ref{ref-8.4-41} 
each element of $\exp(\f{D}^0(A, \mu_0, \eta_0)\otimes m)$ can
be uniquely written as $\exp(0, s^{-d-1}\xi)\exp(f,0)$ for $f \in \bar{\mathfrak{C}}^0(A)\otimes_k m=\Hom(A/k, A)
\tr_k m\subset \Hom(A,A)\otimes_k m$ and $\xi \in \overline{\CC}^{\,-}_{d+1}(A)\tr_k m$.  We put  
$
\phi=e^f
$.
Then $\phi\in \Hom(A,A)\otimes_k R$ is such that $\phi\mod m=\Id_A$.

Assume that
\begin{equation}
\label{ref-8.2-44}
\exp(0,s^{-d-1} \xi)\ast\exp(f,0)\ast(\mu_1, s^{-d-1}\eta_1)= (\mu_2, s^{-d-1}\eta_2)
\end{equation}
We define $\Phi(R)$ on maps as follows
\begin{equation}
\label{ref-8.3-45}
\Phi(R)(\exp(0,s^{-d-1}\xi)\exp(f,0))=(e^f,(-1)^d\xi)
\end{equation}
For this to be well defined we should have a morphism
\[
(\phi , (-1)^d\xi):(\mu_0+\mu_1,\eta_0+\eta_1)\mor 
(\mu_0+\mu_2,\eta_0+\eta_2)
\]
in $\Def^\flat_{A,\eta_0}(R)$. In other words:
\begin{enumerate}
\item[(a)]
  $\phi :(A\tr_k R, -(\mu_0+\mu_1)) \mor (A\tr_k R, -(\mu_0+\mu_2))$ 
is an $R$-algebra morphism;
\item[(b)] 
$\phi(\eta_0+\eta_1)= \eta_0+\eta_2+(-1)^d(L_{\mu_0+\mu_2}+u\mathsf{B})(\xi)$. 
\end{enumerate}
Put $\eta'_i=s^{-d-1}\eta_i$ for $i=0,1,2$, $\xi'=s^{-d-1}\xi$. 
We invoke Lemma \ref{ref-8.5-42} with $m_0=-\eta'_0$. Then \eqref{ref-8.2-44} yields
\begin{equation}
\label{ref-8.4-46}
(\mu_2, \eta'_2)=(\exp(f)\ast\mu_1,e^f(\eta'_0+\eta'_1)-\eta'_0-
L_{\exp(f)\ast\mu_1}(\xi')-(L_{\mu_0}+u\mathsf{B})(\xi'))
\end{equation}
We may compute $\exp(f)\ast\mu_1$ inside unnormalized cochains $\C^\bullet(A)$ and then we 
may invoke lemma \ref{ref-8.6-43}.  We find
\[
\exp(f)\ast\mu_1=e^{\ad f}(\mu_0+\mu_1)-\mu_0
\]
Furthermore a direct computation shows that
\begin{align*}
e^{\ad f}(\mu_0+\mu_1)&=e^f\circ (\mu_0+\mu_1)\circ (e^{-f},e^{-f})\\
&=\phi\circ (\mu_0+\mu_1)\circ (\phi^{-1},\phi^{-1})
\end{align*}
Hence \eqref{ref-8.4-46} translates into
\begin{align*}
\mu_0+\mu_2&=\phi\circ (\mu_0+\mu_1)\circ (\phi^{-1},\phi^{-1})\\
\eta_0'+\eta'_2&=\phi(\eta_0'+\eta'_1)-(L_{\mu_0+\mu_2}+u\mathsf{B})(\xi')
\end{align*}
The first of these equations yields (a). The second yields (b)
taking into account that $L_{\mu_0+\mu_2}+u\mathsf{B}$
has degree one, which induces a sign change.

It remains to show that our assignment respects compositions. By Lemma
\ref{ref-8.4-41} we have for $f,g,h\in \bar{\mathfrak{C}}^0(A)\otimes_k m$
such that $\exp(h)=\exp(g)\exp(f)$,
$\nu,\xi
\in \overline{\CC}_{-d-1}^{\,-}(A)\otimes_k m$:
\begin{align*}
\Phi(R)(\exp(0, s^{-d-1}\nu)&\exp(g,0)\circ \exp(0, s^{-d-1}\xi)\exp(f,0))\\
=& \Phi(R)(\exp(0, s^{-d-1}\nu)\exp(0, s^{-d-1} e^g\xi)\exp(g,0)\exp(f,0))\\
=&\Phi(R)( \exp(0,s^{-d-1}(\nu+e^g\xi))\exp(h,0))\\
=&(e^h, (-1)^d(\nu+e^g\xi))\\
=&(e^ge^f, (-1)^d(\nu+e^g\xi))
\end{align*}
and
\begin{align*}
\Phi(R)(\exp(0, s^{-d-1}\nu)\exp(g,0))&\circ \Phi(R)(\exp(0, s^{-d-1}\xi)\exp(f,0))\\
&=(e^g,(-1)^d\nu)(e^f,(-1)^d\xi)\\
&=(e^ge^f, (-1)^d(\nu+e^g\xi))
\end{align*}
by \eqref{ref-6.1-29}. We conclude that $\Phi(R)$ is indeed a map of groupoids.
By Lemma \ref{ref-8.3-40} it is bijective on objects, and running the above
computation backwards, starting from \eqref{ref-8.3-45}, we see that it
also bijective on maps. Thus $\Phi(R)$ is an isomorphism of groupoids.
\end{proof}
The following result implies Proposition \ref{ref-6.3-31}.
\begin{proposition}
Assume that $\eta_0,\eta'_0\in \overline{\CC}^-_{d}(A)$ induce
the same element in $\HC^-_{d}(A)$. Then $\mathfrak{D}^\bullet(A,\eta_0)\cong
\mathfrak{D}^\bullet(A,\eta_0')$.
\end{proposition}
\begin{proof}
From \eqref{ref-7.3-35} one sees that it is sufficient to show that
$(0,s^{-s-1}\eta_0),(0,s^{-d-1}\eta'_0)$ are in the same $G(\mathfrak{D}^\bullet(A)^\sharp)$
orbit. Pick $\xi\in \overline{\CC}^{\,-}_{d+1}(A)$ such that $\eta_0'=\eta_0+(-1)^d(L_{\mu_0}+u\mathsf{B})\xi$. We
compute using \eqref{ref-7.1-33}
\begin{align*}
\exp(0,s^{-d-1}\xi)\ast (0,s^{-d-1}\eta_0)&=(0,s^{-d-1}\eta_0)-
(0,(L_{\mu_0}+u\mathsf{B})(s^{-d-1}\xi))\\
&=(0,s^{-d-1}\eta_0')\qed
\end{align*}
\def\qed{}\end{proof}

\section{Relation with Hochschild cohomology}
\label{ref-9-47}
Let $(A,\bar{\eta}_0)$ be a $d$-Calabi-Yau $k$-algebra and let $-\mu_0$ be the
multiplication of $A$. Let $(R,m)\in \Nilp$. We may define pseudo-functors
$\Def_A$, $\Def^\flat_{A}:\Nilp\mor \Gd$ in the same way as $\Def_{A,\eta_0}$, $\Def^\flat_{A,\eta_0}$, ignoring $\eta_0$. The induced morphism
\[
\Def^\flat_A(R)\mor \Def_A(R)
\]
is essentially surjective on objects and surjective on morphisms. Furthermore
there is an isomorphism of groupoids
\[
\Phi(R):\MC(\bar{\mathfrak{C}}^\bullet(A)\otimes_k m)\mor
\Def^\flat_A(R):\mu\mapsto \mu_0+\mu
\]
The obvious  morphism of DG-Lie algebras
\[
\phi:\mathfrak{D}^\bullet(A,\eta_0)\mor \bar{\mathfrak{C}}^\bullet(A):(\mu,\eta)\mapsto \mu
\]
makes the following diagram commutative:
\[
\xymatrix{
\MCcal(\mathfrak{D}^\bullet(A,\eta_0))\ar[d]_{\Phi}\ar[r]^\phi& \MCcal(\bar{\mathfrak{C}}^\bullet(A))\ar[d]^{\Phi}\\
\Def_{A,\eta} \ar[r]_{\text{forget $\eta$}}& \Def_A
}
 \]
\section{Homology of {\mathversion{bold} $\mathfrak{D}^\bullet(A,\eta)$}}
Let $(A,\bar{\eta}_0)$ be a $d$-Calabi-Yau algebra as before with multiplication $-\mu_0$.
In this section we prove that the homology of $\f{D}^\bullet(A, \eta_0)$ is isomorphic to 
$\HC^-_{-\b+d-1}(A)$. Furthermore we show that the induced Lie bracket 
on $\HC^-_{-\b+d-1}(A)$ is given Menichi's string topology bracket \cite{Menichi}.

In our statements and computations we will use the following conventions:
\begin{itemize}
\item Taking homology classes is indicated by overlining. 
\item Depending on context $\cong$ will mean either  ``up to homotopy'' (when
discussing maps) or ``up to
addition of a coboundary'' (when discussing elements).
\end{itemize}
\begin{theorem}
\label{ref-10.1-48}
The map
\[
\Psi: \f{D}^\b(A,\eta_0) \mor \Sigma^{-d+1}
\overline{\mathfrak{\CC}}^{\,-}_\b(A): (\mu, s^{-d-1}\eta) 
\fun (-1)^{\vert \mu\vert -1}(i_\mu+uS_\mu)(s^{-d+1}\eta_0)+us^{-d+1}\eta
\]
is a quasi-isomorphism of complexes. 
\end{theorem}
\begin{proof}
To simplify the notation we put
\[
I_\mu=i_\mu+uS_\mu
\]
We first check that $\Psi$ does indeed commute with differentials. 
Write $\eta_0'=s^{-d-1}\eta_0$, $\eta'=s^{-d-1}\eta$. Then
\begin{equation}
\label{ref-10.1-49}
\Psi(\mu,\eta')=s^2((-1)^{\vert \mu\vert -1}I_\mu\eta'_0+u\eta')
\end{equation}
and hence
\begin{align*}
(d\circ \Psi)(\mu,\eta')&=
(\mathsf{b}+u\mathsf{B})s^2((-1)^{\vert \mu\vert -1}I_\mu \eta'_0+u\eta')\\
&= s^2((-1)^{\vert \mu \vert-1}(\mathsf{b}+u\mathsf{B}) I_\mu\eta'_0+u(\mathsf{b}+u\mathsf{B})\eta')\\
&=  s^{2}((-1)^{\vert \mu \vert-1}[\mathsf{b}+u\mathsf{B}, I_\mu](\eta'_0)+u(\mathsf{b}+u\mathsf{B})\eta')&&(\text{since $(\mathsf{b}+u\mathsf{B})\eta'_0=0$)}\\
&= s^{2}((-1)^{\vert \mu \vert-1}(uL_{\mu}-I_{d\mu})\eta'_0+u(\mathsf{b}+
u\mathsf{B})\eta')&& (\text{by \eqref{ref-4.8-13}})\\
&=  s^{2}((-1)^{\vert \mu \vert}I_{d\mu}\eta'_0+
u((\mathsf{b}+u\mathsf{B})\eta'-(-1)^{\vert \mu \vert}L_\mu \eta'_0))\\
&=\Psi(d\mu, (\mathsf{b}+u\mathsf{B})\eta'-(-1)^{\vert \mu \vert}L_\mu \eta'_0)
&&(\text{by \eqref{ref-10.1-49}})
\\
&=(\Psi\circ d)(\mu, \eta')
\end{align*}
To see that $\Psi$ is indeed a quasi-isomorphism, consider the following commutative diagram
\begin{displaymath}
\xymatrix{
0\ar[r] &  \Sigma^{-d-1}\overline{\CC}_\b^{\,-}(A)\ar[r] \ar[d]_\Psi& \f{D}^\bullet(A,\eta_0) \ar[r]\ar[d]\ar[d]_\Psi& \bar{\f{C}}^\b(A) \ar[r]\ar[d]^{\overline{\Psi}}&0\\
0\ar[r] & u\Sigma^{-d+1}\overline{\CC}_\b^{\,-}(A)\ar[r] & \Sigma^{-d+1}
\overline{\CC}_\b^{\,-}(A) \ar[r]& \Sigma^{-d+1}\bar{\C}_\b(A) \ar[r]&0
}
\end{displaymath}
The map on the left is  multiplication by $u$ which is an
isomorphism.  The map~$\overline{\Psi}$ is given on cohomology by
\[
\overline{\mu}\mapsto\pm\overline{I_\mu\eta_0\,\operatorname{mod} u}= \pm\overline{i_{\mu} \pi(\eta_0)}
\]
where $\pi$ is as in Proposition~\ref{ref-5.7-25}. Hence $\overline{\Psi}$ 
is an isomorphism by Proposition~\ref{ref-5.5-20}. From the five lemma we conclude that the middle arrow is an isomorphism on cohomology as well.
\end{proof}

We now describe the Lie bracket on $\HC^-_\bullet(A)$ induced by the quasi-isomorphism~$\Psi$.
As already used in the above proof 
the map 
\[
-\cap \pi(\bar{\eta}_0): \HH^i(A)
\mor \HH_{d-i}(A) 
\]
is invertible by Proposition \ref{ref-5.5-20}.  Let us denote its inverse by
$j$. Using $j$, one can transport the cup product on $\HH^\b(A)$ to a
product on $\HH_\b(A)$ 
\[
\cdot:\HH_i(A)\times \HH_j(A)\mor \HH_{i+j-d}(A)
\]
with explicit formula
\[
a\cdot b=(j(a)\cup j(b))\cap \pi(\bar{\eta}_0)
\]
or in a form more suitable for us below
\begin{equation}
\label{ref-10.2-50}
i_{\mu_1}\pi(\bar{\eta}_0)\cdot i_{\mu_2}\pi(\bar{\eta}_0)=i_{\mu_1\cup \mu_2}\pi(\bar{\eta}_0)
\end{equation}
\begin{theorem}\label{ref-10.2-51}
The Lie bracket induced on 
\[
H^\bullet(\Sigma^{-d+1}\overline{\CC}^{\,-}_\bullet(A))=
\HC_{\bullet+d-1}(A)
\]
by the quasi-isomorphism $\Psi$ is given by
\[
[-,-]: 
\HC^-_n(A) \times \HC^-_{m}(A) \mor \HC^-_{n+m-d+1}(A): (\eta_1, \eta_2) 
\fun (-1)^{|\eta_1|+d}\mathsf{B}(\pi(\eta_1)\,\cdot\,\pi(\eta_2))
\]
where $\mathsf{B}$ is given by
\[
\mathsf{B}:\HH_q(A)\mor \HC_{q+1}^-(A):\bar{\nu}\mapsto \overline{\mathsf{B}\nu}
\]
\end{theorem}
We first need the following technical lemma.
\begin{lemma}
\label{ref-10.3-52}
Let $\mu \in \bar{\f{C}}^\b(A)$ and $\eta \in \overline{\CC}_\b^{\,-}(A)$ 
be cocycles. Then $L_\mu \eta$ and $\mathsf{B}i_\mu \pi(\eta)$ are both cocycles 
in $\overline{\CC}_\b^{\,-}(A)$ 
and $\overline{\mathsf{B}i_\mu \pi(\eta})=\overline{L_\mu \eta}$ in $\HC^-_\bullet(A)$.
\end{lemma}
\begin{proof}
$L_\mu \eta$ is a cocycle by \eqref{ref-4.7-12}. $\mathsf{B}i_\mu \pi(\eta)$ is a cocycle since $\pi(\eta)$ is a cocycle in $\bar{\C}_\bullet(A)$
and 
\[
(\mathsf{b}+u\mathsf{B})(\mathsf{B}i_\mu \pi(\eta))= 
\mathsf{b}\mathsf{B}i_\mu \pi(\eta)= -\mathsf{B} \mathsf{b}i_\mu \pi(\eta)=0
\]
where the last equality follows from \eqref{ref-4.5-10}.

For the second claim, we first multiply by $u$:
\begin{align*}
u(L_\mu\eta-\mathsf{B}i_\mu\pi(\eta))&=[\mathsf{b}+u\mathsf{B},I_\mu]\eta-u\mathsf{B}i_\mu\pi(\eta)&&(\text{by \eqref{ref-4.8-13}})\\
&=(\mathsf{b}+u\mathsf{B})I_\mu\eta-u\mathsf{B}i_\mu\pi(\eta)&&(\text{since $(\mathsf{b}+u\mathsf{B})\eta=0$})\\
&=(\mathsf{b}+u\mathsf{B})(I_\mu\eta-i_\mu \pi(\eta))&& (\text{since $\mathsf{b}i_\mu\pi(\eta)=0$})
\end{align*}

Now, $\pi(I_\mu\eta-i_\mu(\pi(\eta))= i_\mu\pi(\eta)-i_\mu\pi(\eta)=0$, which means that $I_\mu\eta-i_\mu\pi(\eta)$ is divisible by $u$. Thus it follows that 
\[
L_\mu\eta-\mathsf{B}i_\mu\pi(\eta)=(\mathsf{b}+u\mathsf{B})(u^{-1}(I_\mu\eta-i_\mu\pi(\eta)))
\] hence the claim.
\end{proof}

\begin{proof}[Proof of Theorem \ref{ref-10.2-51}]
Let $(\mu_1, s^{-d-1}\eta_1)$ and $(\mu_2,s^{-d-1}\eta_2)$ be two cocycles in $\f{D}^\b(A, \eta_0)$. We
must prove for $\eta_i'=s^{-d-1}\eta_i$
\begin{equation}
\label{ref-10.3-53}
\overline{s^{d-1}\Psi([(\mu_1, \eta'_1),(\mu_2,\eta'_2)])}=[\overline{s^{d-1}\Psi(\mu_1,\eta'_1)},\overline{s^{d-1}\Psi(\mu_2,\eta'_2)}]
\end{equation}
We will first compute the lefthand side of \eqref{ref-10.3-53}.
Writing out the differential in $\mathfrak{D}^\bullet(A,\eta'_0)$ explicitly, the fact that $({\mu_1}, {\eta'_1})$, $(\mu_2,{\eta'_2})$
are cocycles implies
\begin{equation}
\label{ref-10.4-54}
\begin{gathered}
d{\mu_1}=d\mu_2=0\\
(\mathsf{b}+u\mathsf{B}){\eta'_1}-(-1)^{\vert {\mu_1} \vert}L_{\mu_1} {\eta'_0}=(\mathsf{b}+u\mathsf{B}){\eta'_2}-(-1)^{\vert \mu_2\vert }L_{\mu_2}{\eta'_0}=0
\end{gathered}
\end{equation}
where ${\eta'_0}=s^{-d-1}\eta'_0$. We compute
\begin{equation}
\begin{aligned}
 x&\overset{\text{def}}{=}s^{d-1}\Psi([({\mu_1},{\eta'_1}),(\mu_2{\eta'_2})])\\
 &= s^{d-1}\Psi([{\mu_1},\mu_2], L_{\mu_1}{\eta'_2}-(-1)^{\vert {\mu_1}\vert \vert {\eta'_2}  \vert} L_{\mu_2} {\eta'_1})\\
&=s^{d+1}((-1)^{|{\mu_1}|+|\mu_2| -1}I_{[{\mu_1},\mu_2]} {\eta'_0}+u( L_{\mu_1}{\eta'_2}-(-1)^{\vert {\mu_1}\vert \vert \mu_2  \vert} L_{\mu_2} {\eta'_1})) \label{ref-10.5-55}
\end{aligned}
\end{equation}
where we have used \eqref{ref-10.1-49} and the fact that $|{\eta'_2}|=|\mu_2|$.

We now consider the boundary element $(\mathsf{b}+u\mathsf{B})I_{\mu_1} {\eta'_2}$. By \eqref{ref-4.8-13}, we have
\[
(\mathsf{b}+u\mathsf{B})I_{\mu_1} {\eta'_2}-(-1)^{\vert{\mu_1} \vert+1} I_{\mu_1} (\mathsf{b}+u\mathsf{B}) {\eta'_2}+I_{d{\mu_1}}{\eta'_2}=uL_{\mu_1} {\eta'_2}
\]
Taking into account \eqref{ref-10.4-54} this becomes
\begin{align*}
(\mathsf{b}+u\mathsf{B})I_{\mu_1} {\eta'_2}
&= (-1)^{\vert {\mu_1}\vert-1 } I_{\mu_1}(\mathsf{b}+u\mathsf{B}) {\eta'_2}+uL_{\mu_1} {\eta'_2}\\
&=(-1)^{\vert {\mu_1}\vert-1  +\vert \mu_2 \vert} I_{\mu_1}  L_{\mu_2}{\eta'_0}+uL_{\mu_1} {\eta'_2}
\end{align*}
and similarly
\[
(\mathsf{b}+u\mathsf{B})I_{\mu_2}{\eta'_1}=(-1)^{\vert \mu_2\vert -1  + \vert {\mu_1} \vert } I_{\mu_2} L_{{\mu_1}}{\eta'_0}+uL_ {\mu_2} {\eta'_1}
\]

We now subtract both boundaries with appropriate sign from \eqref{ref-10.5-55} to obtain the following homologous cocycle
\begin{equation}
\begin{aligned}
\label{ref-10.6-56}
  x&\cong
  (-1)^{\vert {\mu_1}\vert+\vert \mu_2\vert
    -1} 
s^{d+1}(I_{[{\mu_1},\mu_2]}{\eta'_0}-
    I_{\mu_1} L_{\mu_2}{\eta'_0}+ (-1)^{|{\mu_1}||\mu_2|}I_{\mu_2} L_{{\mu_1}}
  {\eta'_0})\\
&=(-1)^{\vert {\mu_1}\vert+\vert \mu_2\vert
    -1} 
s^{d+1}
(
I_{[{\mu_1},\mu_2]}-
    I_{\mu_1} L_{\mu_2}+ (-1)^{|{\mu_1}||\mu_2|}I_{\mu_2} L_{{\mu_1}}
){\eta'_0}
\end{aligned}
\end{equation}
By Lemma \ref{ref-4.2-14} and \eqref{ref-10.4-54}:
\[
[L_{\mu_1}, I_{\mu_2}]-(-1)^{\vert {\mu_1} \vert} I_{[{\mu_1}, \mu_2]}\cong 0
\]
Thus 
\begin{align*}
I_{[{\mu_1}, \mu_2]}&\cong (-1)^{\vert {\mu_1}\vert}(L_{\mu_1} I_{\mu_2}-(-1)^{\vert L_{\mu_1} \vert \vert I_{\mu_2}\vert }I_{\mu_2}  L_{{\mu_1}})\\
&=(-1)^{\vert {\mu_1} \vert }( L_{\mu_1}  I_{\mu_2}-(-1)^{\vert {\mu_1} \vert (\vert \mu_2 \vert +1)}I_{\mu_2} L_{\mu_1})\\
&=(-1)^{\vert {\mu_1} \vert } L_{\mu_1}  I_{\mu_2}-(-1)^{\vert {\mu_1} \vert \vert \mu_2 \vert}I_{\mu_2} L_{\mu_1}
\end{align*}
Substituting this in  \eqref{ref-10.6-56} we obtain
\begin{equation}
\label{ref-10.7-57}
\nonumber x\cong 
(-1)^{\vert {\mu_1}\vert+\vert \mu_2\vert
    -1} 
s^{d+1}((-1)^{\vert {\mu_1} \vert } L_{\mu_1}  I_{\mu_2}\eta'_0
-
    I_{\mu_1} L_{\mu_2}{\eta'_0})
\end{equation}
Next we observe, using \eqref{ref-4.8-13}
\begin{equation}
\label{ref-10.7-58}
\begin{split}
[\mathsf{b}+u\mathsf{B},I_{\mu_1} I_{\mu_2}-&(-1)^{(|{\mu_1}|+1)(|\mu_2|+1)}I_{\mu_2\cup{\mu_1}}]\\
&=[\mathsf{b}+u\mathsf{B},I_{\mu_1}]I_{\mu_2}
+(-1)^{|{\mu_1}|+1}I_{\mu_1}[\mathsf{b}+u\mathsf{B},I_{\mu_2}]-(-1)^{(|{\mu_1}|+1)(|\mu_2|+1)}[\mathsf{b}+u\mathsf{B},I_{{\mu_2}\cup \mu_1}]\\
&=u(L_{\mu_1} I_{\mu_2}+(-1)^{|{\mu_1}|+1}I_{\mu_1} L_{\mu_2}-(-1)^{(|{\mu_1}|+1)(|\mu_2|+1)}L_{\mu_2\cup {\mu_1}})
\end{split}
\end{equation}
and also using \eqref{ref-4.2-7}:
\begin{align*}
I_{\mu_1} I_{\mu_2}-(-1)^{(|{\mu_1}|+1)(|\mu_2|+1)}I_{\mu_2\cup{\mu_1}}\mod u&=
i_{\mu_1} i_{\mu_2}-(-1)^{(|{\mu_1}|+1)(|\mu_2|+1)}i_{\mu_2\cup{\mu_1}}\\
&=0
\end{align*}
In other words $I_{\mu_1} I_{\mu_2}-(-1)^{(|{\mu_1}|+1)(|\mu_2|+1)}I_{\mu_2\cup{\mu_1}}$ is divisible by
$u$ and we obtain from \eqref{ref-10.7-58} 
\[
L_{\mu_1} I_{\mu_2}+(-1)^{|{\mu_1}|+1}I_{\mu_1} L_{\mu_2}\cong (-1)^{(|{\mu_1}|+1)(|\mu_2|+1)}L_{\mu_2\cup {\mu_1}}
\]
Substituting this back in \eqref{ref-10.7-57} we find
\begin{align*}
x&\cong  (-1)^{|{\mu_1}|(|\mu_2|+1)}
s^{d+1}
L_{\mu_2\cup {\mu_1}}
{\eta'_0}\\
&= (-1)^{|{\mu_1}|(|\mu_2|+1)}s^{d+1}\mathsf{B}i_{\mu_2\cup{\mu_1}} \pi({\eta'_0})
&& (\text{by Lemma \ref{ref-10.3-52}})\\
&\cong(-1)^{|\mu_2|+1}s^{d+1}\mathsf{B}i_{\mu_1\cup{\mu_2}} \pi({\eta'_0})&&(\text{by \S\ref{ref-4.2-4}})\\
&\cong (-1)^{|\mu_2|+1} (-1)^{(|\mu_1|+|\mu_2|+1)(d+1)}\mathsf{B}i_{\mu_1\cup{\mu_2}} \pi({\eta_0})
\end{align*}
and hence by \eqref{ref-10.2-50}
\[
\bar{x}=(-1)^{|\mu_2|+1} (-1)^{(|\mu_1|+|\mu_2|+1)(d+1)}\mathsf{B}(i_{\mu_1}\pi(\bar{\eta}_0)\cdot i_{\mu_2}\pi(\bar{\eta}_0))
\]

To compute the righthand side of \eqref{ref-10.3-53} we note
\begin{align*}
\pi(s^{d-1}\Psi(\mu_i,\eta_i'))&=\pi(s^{d+1}((-1)^{|\mu_i|-1}I_{\mu_i}\eta'_0+u\eta'_i)) &&(\text{by \eqref{ref-10.1-49}})\\
&=(-1)^{|\mu_i|-1}(-1)^{(|\mu_i|+1)(d+1)}i_{\mu_i}\pi(\eta_0)
\end{align*}
so that
\begin{align*}
[\overline{s^{d-1}\Psi(\mu_1,\eta'_1)},\overline{s^{d-1}\Psi(\mu_2,\eta'_2)}]&=(-1)^{|\mu_1|+d}\mathsf{B}
(\overline{s^{d-1}\Psi(\mu_1,\eta'_1)}\cdot \overline{s^{d-1}\Psi(\mu_1,\eta'_1)})\\
&=(-1)^{|\mu_1|+d+|\mu_1|+|\mu_2|}(-1)^{(|\mu_1|+|\mu_2|)(d+1)}\mathsf{B}(\overline{i_{\mu_1}\pi(\eta_0)}\cdot \overline{i_{\mu_2}\pi(\eta_0)})\\
&=(-1)^{|\mu_2|+1} (-1)^{(|\mu_1|+|\mu_2|+1)(d+1)}\mathsf{B}(i_{\mu_1}\pi(\bar{\eta}_0)\cdot i_{\mu_2}\pi(\bar{\eta}_0))
\end{align*}
finishing the proof.
\end{proof}

\section{The commutative case}
\label{ref-11-59}
In this section we will use formality results from \cite{Dolgushev,Ko3,Shoikhet2,Tsygan,Willwacher1} 
so we will assume that the ground field
$k$ contains $\mathbb{R}$. It is likely that it is sufficient to assume $\operatorname{char} k=0$ by using the
 methods from \cite{dtt2,dtt} but we have not checked it.

Let $A=\mathcal{O}(X)$ where $X$ is a smooth affine $d$-dimensional 
Calabi-Yau variety over~$k$. Let $T^{\poly,\bullet}(A)$ be the Lie algebra of poly-vector fields on $X$. We assume that
$T^{\poly,\bullet}(A)$ is shifted in such a way that the Lie bracket has degree zero. 
Similarly let $\Omega^\bullet(A)$ be the differential forms on $X$ (not shifted). 

We fix  a volume form $\eta\in \Omega^d(A)$. The Hochschild-Kostant-Rosenberg-map furnishes an isomorphism
$\HH_d(A)\cong \Omega^d(A)$. So we may consider $\eta$ as an element
in $\HH_d(A)$ and hence by Proposition \ref{ref-5.7-25} as a cycle (still
denoted by $\eta$) in $\CC^-_d(A)$. It is well-known and easy to
see that $(A,\eta)$ is a Calabi-Yau algebra in the sense of Ginzburg.
Let
\[
\div:T^{\poly,\bullet}(A)\mor T^{\poly,\bullet-1}(A)
\]
be the divergence operator  corresponding to $\eta$ (see \S\ref{ref-11.4-68} below).
The divergence is a differential which acts as a derivation with respect
to the Lie bracket on $T^{\poly,\bullet}(A)$.

In this section we will prove the following result
\begin{theorem} 
\label{ref-11.1-60}
There exists an isomorphism
\[
(T^{\poly,\bullet}(A)[[u]],-u\div)\cong \mathcal{D}^\bullet(A,\eta)
\]
in the homotopy category of DG-Lie algebras which fits in a diagram like \eqref{ref-1.3-2}.
\end{theorem}
Recall that the homotopy category of DG-Lie algebras is the category of
DG-Lie algebras with quasi-isomorphisms formally inverted.

\subsection{Semi-direct products for  $L_\infty$-algebras}
\label{ref-11.1-61}
We remind the reader of a few basic definition regarding $L_\infty$-algebras and modules.
Let $\mathfrak{h}^\bullet$ be a graded $k$-vector space. Recall that
an $L_\infty$-structure on $\mathfrak{h}^\bullet$ is a square zero,
degree one coderivation $Q$ on the symmetric coalgebra $S^c(\Sigma
\mathfrak{h}^\bullet)$.  Such an $L_\infty$-structure is determined by
its Taylor coefficients $\partial^n Q$ which are maps $S^n(\Sigma
\mathfrak{h}^\bullet)\mor \Sigma \mathfrak{h}^\bullet$.  \emph{Here
  and in related situations below we always assume that zeroth order
  Taylor coefficient are zero.}

A DG-Lie algebra can be made into an $L_\infty$-algebra
by putting $\partial^1Q(sg)=-sdg$, $\partial^2Q(sg,sh)=(-1)^{|g|}s[g,h]$,
$\partial^nQ=0$ for $n\ge 3$.

A morphism of $L_\infty$-algebras $\psi:(\mathfrak{g}^\bullet,Q)\rightarrow(\mathfrak{h}^\bullet,Q)$    
is a coalgebra morphism
$\psi:S^c(\Sigma\mathfrak{g}^\bullet)\mor S^c(\Sigma\mathfrak{h}^\bullet)$
commuting with $Q$.  It is also determined by its Taylor coefficients
$\partial^n\psi:S^n(\Sigma\mathfrak{g}^\bullet)\mor\Sigma
\mathfrak{h}^\bullet$.

If $V^\bullet$ is a graded $k$-vector space then an
$L_\infty$-$\mathfrak{h}^\bullet$-module structure on $V^\bullet$ is a 
square zero, degree one differential  $R:S^c(\Sigma \mathfrak{h}^\bullet)\otimes V^\bullet
\rightarrow S^c(\Sigma \mathfrak{h}^\bullet)\otimes V^\bullet$  satisfying
\[
(Q\otimes \Id_{S^c\mathfrak{h}}\otimes \Id_V+\Id_{S^c\mathfrak{h}}\otimes R)\circ(\Delta\otimes \Id_V)= (\Delta\otimes \Id_V)\circ R
\]
as maps from $S^c(\Sigma\mathfrak{h}^\bullet)\otimes V^\bullet$ to
$S^c(\Sigma\mathfrak{h}^\bullet)\otimes S^c(\Sigma\mathfrak{h}^\bullet) \otimes V^\bullet$
An $L_\infty$-$\mathfrak{h}^\bullet$-module structure $R$ on
$V^\bullet$ is entirely determined by the maps
$\partial^{n+1}R:S^n(\Sigma\mathfrak{h}^\bullet)\otimes V^\bullet\mor
V^\bullet$. If $\mathfrak{h}^\bullet$~is a DG-Lie algebra and
$V^\bullet$ is a DG-module over it then $V^\bullet$ can be made
into an $L_\infty$-module over $\mathfrak{h}$ by putting $\partial^1R(v)=dv$, $\partial^2R(sg,v)=g\cdot v$, $\partial^nR=0$ for $n\ge 3$. 

If $V^\bullet$ is an
$L_\infty$-$\mathfrak{h}^\bullet$-module then so are $\Sigma^m V^\bullet$ for
all $m$ using the obvious sign convention $\partial^{n+1}R(sg_1,\ldots,sg_n,s^mv)=(-1)^{m(n+|g_1|+\cdots+|g_n|)}\partial^{n+1}R(sg_1,\ldots,sg_n,v)$.
We may combine the $L_\infty$-structures on $\mathfrak{h}^\bullet$ and
$\Sigma V^\bullet$ to make the direct sum $\mathfrak{h}^\bullet\oplus 
V^\bullet$ into an $L_\infty$-algebra. We will denote the resulting $L_\infty$-algebra
by $\mathfrak{h}^\bullet\ltimes V^\bullet$ and call it the \emph{semi-direct product} of $\mathfrak{h}^\bullet$.
This is an obvious generalization of the semi-direct product of a
DG-Lie algebra with a DG-module which was used in \S\ref{ref-8-36}.

Assume that $(V^\bullet,R)$, $(W^\bullet,R)$ are
$L_\infty$-$\frak{h}^\bullet$-modules. An $L_\infty$ morphism $\mu:V^\bullet\mor
W^\bullet$ is a comodule map $\mu: S^c(\Sigma \mathfrak{g})\otimes V^\bullet
\mor S^c(\Sigma \mathfrak{g})\otimes W^\bullet$.
commuting with $R$. It is determined
by its Taylor coefficients $\partial^n \mu: S^n(\Sigma \mathfrak{h}^\bullet)
\otimes V^\bullet\mor  W^\bullet$.

Given in addition an $L_\infty$-morphism 
$\psi:\mathfrak{g}^\bullet\mor \mathfrak{h}^\bullet$
 the pullback $V_\psi^\bullet$ of $V^\bullet$ is defined as follows:
\begin{multline*}
\partial^{n+1} R_{\psi}(sg_1,\ldots,sg_n,v)
=\\\sum_{t,1\le m_1<\cdots < m_{t-1}<n} 
\pm \partial^{t+1}R(\partial^{m_1}\psi(sg_{i_1},\ldots,sg_{i_{m_1}}),
\partial^{m_2-m_1}\psi(sg_{i_{m_1+1}},\ldots,sg_{i_{m_2}}),\\
\ldots,
\partial^{n-m_{t-1}}\psi(sg_{i_{m_{t-1}+1}},\ldots,sg_{n}),v)
\end{multline*}
where for all $j$: $i_{m_j+1}<\cdots<i_{m_{j+1}}$ and the sign is the
Koszul sign of the corresponding shuffle of the $(sg_i)_i$.
By construction we have a canonical $L_\infty$-morphism 
\[
 \psi_V:\mathfrak{g}^\bullet\ltimes V^\bullet_\psi \mor \mathfrak{h}^\bullet\ltimes V^\bullet
\]
which restricted to $S^n(\Sigma \mathfrak{g})$ coincides with $\partial^n\psi$.

\subsection{Twisting}
\label{ref-11.2-62}
Assume that $\psi:\frak{g}^\bullet\mor \frak{h}^\bullet$ is a $L_\infty$-morphism
between $L_\infty$-algebras equipped with some type of topology and
let $\omega\in \frak{g}^1$ be a Maurer-Cartan element in $\mathfrak{g}^1$, i.e.\ a solution of the $L_\infty$-Maurer-Cartan equation
\[
\sum_{i\ge 1} \frac{1}{i!} (\partial^i Q)(\underbrace{\omega\cdots\omega}_i)=0
\]
One has to  assume that one is in a situation
where all occurring series are convergent and standard series
manipulations are allowed.  In our application below the series are in fact finite.

 Define $Q_\omega$, $\psi_\omega$ and $\omega'$  by  \cite{ye3}
\begin{align}
\label{ref-11.1-63}
(\partial^i Q_\omega)(\gamma)&=\sum_{j\ge 0} \frac{1}{j!} (\partial^{i+j} Q)
(\underbrace{\omega\cdots\omega}_j \gamma)\qquad \text{(for $i>0$)}\\
\label{ref-11.2-64}
(\partial^i \psi_\omega)(\gamma)&=\sum_{j\ge 0} \frac{1}{j!} (\partial^{i+j} \psi)
(\underbrace{\omega\cdots \omega}_j \gamma)\qquad \text{(for $i>0$)}\\
\label{ref-11.3-65}
\omega'&=\sum_{j\ge 1} \frac{1}{j!} (\partial^{j} \psi)
(\underbrace{\omega\cdots\omega}_j )
\end{align}
for $\gamma\in S^i(\Sigma\frak{g}^\bullet)$. Then e.g.\ by \cite{ye3}
$\omega'$ is a solution of the Maurer-Cartan equation on $\frak{h}^\bullet$
and  furthermore $\frak{g}^\bullet$, $\frak{h}^\bullet$, when equipped with
$Q_\omega$, $Q_{\omega'}$ are again $L_\infty$-algebras.  Let us
denote these by $\frak{g}^\bullet_\omega$ and $\frak{h}^\bullet_{\omega'}$.  Finally
 $\psi_\omega$ is an $L_\infty$ map $\frak{g}^\bullet_\omega\mor
\frak{h}_{\omega'}^\bullet$.

\subsection{Applying formality to {\mathversion{bold} $\mathfrak{D}(A,\eta)$}}
By  \cite{CF,Ko3,ye3} there
is an $L_\infty$-quasi-isomorphism 
\[
\mathfrak{U}:T^{\poly,\bullet}(A)\rightarrow \bar{\mathfrak{C}}^\bullet(A)
\]
such that $\partial^1 \mathfrak{U}$ is the standard Hochschild-Kostant-Rosenberg quasi-isomorphism.

View $(\overline{\CC}^{\,-}_\bullet(A),\mathsf{b}+u\mathsf{B})$ as an $L_\infty$-module over $T^{\poly,\bullet}(A)$ via $\mathfrak{U}$ 
as in \S\ref{ref-11.1-61}.
We also view $(\Omega^\bullet(A)[[u]],ud)$
as a DG-Lie module over $T^{\poly,\bullet}(A)$ via the Lie derivative.
Then by 
\cite{Dolgushev,Shoikhet2,Tsygan,Willwacher1} 
there is an $L_\infty$-quasi-morphism
of $L_\infty$-modules over $T^{\poly,\bullet}(A)$
\[
\mathfrak{S}:(\overline{\CC}^{\,-}_\bullet(A),\mathsf{b}+u\mathsf{B})\mor (\Omega^\bullet(A)[[u]],ud)
\]
where $\partial^1\mathfrak{S}$ is again the HKR quasi-isomorphism.
Thus we get a roof of $L_\infty$-quasi-morphisms of graded DG-Lie algebras 
\begin{equation}
{\tiny
\label{ref-11.4-66}
\xymatrix{%
& T^{\poly,\bullet}(A)\ltimes \Sigma^{-d-1}\overline{\CC}^{\,-}_\bullet(A)\ar[dl]_{\mathfrak{S}}\ar[dr]^{\mathfrak{U}}&\\
T^{\poly,\bullet}(A)\ltimes \Sigma^{-d-1}\Omega^\bullet(A)[[u]] &&
\bar{\mathfrak{C}}^{\bullet}(A)\ltimes \Sigma^{-d-1}\overline{\CC}^{\,-}_\bullet(A) 
}
}
\end{equation}

We obtain a new roof by twisting with $(0,\eta')$ where $\eta'=s^{-d-1}\eta$. 
\begin{equation}
\label{ref-11.5-67}
{\tiny
\xymatrix{%
& (T^{\poly,\bullet}(A)\ltimes \Sigma^{-d-1}\overline{\CC}^{\,-}_\bullet(A)\ar[dl]_{\mathfrak{S}_{(0,\eta')}}\ar[dr]^{\mathfrak{U}_{(0,\eta')}})_{(0,\eta')}&\\
\mathfrak{T}^\bullet(A,\eta)
&&
\mathfrak{D}^\bullet(A,\eta)
}
}
\end{equation}
where
\[
\mathfrak{T}^\bullet(A,\eta)=(T^{\poly,\bullet}(A)\ltimes \Sigma^{-d-1}\Omega^\bullet(A)[[u]])_{(0,\eta')}
\]
The complexes here are are $2$-step filtered.
The arrows are quasi-isomorphisms since if we take the associated graded
complexes for the $2$-step filtrations 
  we find the same arrows as in \eqref{ref-11.4-66}.
\subsection{Divergence etc\ldots}
\label{ref-11.4-68}
The divergence operator is defined by
\[
\div: T^{\bullet,\poly}(A)\mor T^{\bullet-1,\poly}(A)
\]
via the following identity
\[
d(\gamma \cap \eta)=\div \gamma \cap \eta
\]
We conclude immediately
\[
\div^2=0
\]
and furthermore  the following is true
\cite{Schechtman}:
\[
(-1)^{|\gamma_1|}[\gamma_1,\gamma_2]=\div(\gamma_1\gamma_2)-\div(\gamma_1)\gamma_2-(-1)^{|\gamma_1|+1}\gamma_1\div\gamma_2
\]
So $(T^{\poly,\bullet}(A),-\div,\cup)$ is a BV-algebra (see App.\ \ref{ref-A-76}). 
\begin{proposition} 
\label{ref-11.2-69}
There is an $L_\infty$-isomorphism of DG-Lie algebras
\[
\delta:(T^{\poly,\bullet}(A)[[u]],-u\div)\mor
\mathfrak{T}^\bullet(A)
\]
\end{proposition}
\begin{proof}
According to Proposition \ref{ref-A.4-81} there exists an $L_\infty$-isomorphism
\[
(T^{\poly,\bullet}(A)[[u]],-u\div)\mor (T^{\poly,\bullet}(A)\ltimes \mathfrak{a},-u\div)
\]
where $\mathfrak{a}$ is the abelian graded Lie algebra  on the
vector space $u T^{\poly,\bullet}(A)[[u]]$. The action of $T^{\poly,\bullet}(A)$ on
$\mathfrak{a}$ is given by
\[
\gamma\star a=[\gamma,a]+(-1)^{|\gamma|}\div \gamma\cup a
\]
To finish the proof it is sufficient to show that the following map
\[
\delta':(T^{\poly,\bullet}(A)\ltimes \mathfrak{a},-u\div)\mor \mathfrak{T}^\bullet(A,\eta)=(T^{\poly,\bullet}(A)\ltimes \Sigma^{-d-1}\Omega^\bullet(A)[[u]])_{(0,\eta')}
:(\gamma,a)\mapsto (\gamma,(-1)^{|a|}u^{-1}a\cap \eta')
\]
is an isomorphism of DG-Lie algebras.
First we show that 
\[
\delta': \mathfrak{a}\mor \Sigma^{-d-1}\Omega^\bullet(A)[[u]]:a\mapsto (-1)^{|a|}u^{-1}(a\cap \eta')
\]
is compatible with the action 
of $T^{\poly,\bullet}(A)$. We compute for $\gamma\in T^{\poly,\bullet}(A)$ and
$a\in \mathfrak{a}$.
\begin{align*}
\delta'(\gamma\star {a})&=\delta'([\gamma,{a}]+(-1)^{|\gamma|}\div \gamma\cup {a})\\
&=(-1)^{|\gamma|+|{a}|}u^{-1}([\gamma,{a}]+(-1)^{|\gamma|}\div \gamma\cup {a})\cap \eta'\\
&=(-1)^{|\gamma|+|{a}|}u^{-1}((-1)^{|\gamma|}\div(\gamma\cup {a})-(-1)^{|\gamma|}\div(\gamma)\cup
{a}+\gamma\cup
\div ({a})+(-1)^{|\gamma|}\div (\gamma)\cup {a})\cap \eta'
\\
&=(-1)^{|\gamma|+|{a}|}u^{-1}((-1)^{|\gamma|}\div (\gamma\cup {a})+\gamma\cup \div{a})\cap \eta'\\
&=(-1)^{|\gamma|+|{a}|}u^{-1} ((-1)^{|\gamma|}d(\gamma\cap({a}\cap \eta'))+\gamma \cap d({a}\cap \eta'))\\
&=L_\gamma(\delta'({a}))
\end{align*}
Now we check compatibility with the differential of $\delta'$
on an element $a\in\mathfrak{a}$.
\begin{align*}
\delta'(-u\div a)&=-(-1)^{|a|+1}\div a\cap \eta'\\
&=(-1)^{|a|}d(a\cap \eta')\\
&=d(\delta'(a))
\end{align*}
Finally we check compatibility  with the differential of $\delta'$
on $\gamma\in T^{\poly,\bullet}(A)$. 
\begin{align*}
\delta'(-u\div \gamma)&=-(-1)^{|\gamma|+1}\div \gamma\cap \eta'\\
&=(-1)^{|\gamma|}d(\gamma\cap \eta')\\
&=(-1)^{|\gamma|}L_a\eta'\\
&=[(0,\eta'),(\gamma,0)]\qed
\end{align*}
\def\qed{}\end{proof}
\begin{proof}[Proof of Theorem \ref{ref-11.1-60}] It suffices to combine diagram \eqref{ref-11.5-67} with Proposition \ref{ref-11.2-69}, taking into account that an $L_\infty$-quasi-isomorphism yields
an isomorphism in the homotopy category of DG-Lie algebras via the bar cobar construction.
\end{proof}
\def\BVm{\operatorname{BV}_-}
\section{Obstructions}
Let $\mathfrak{g}^\bullet$ be a DG-Lie algebra and let
$(S,n)\rightarrow (R,m)$ be a surjective morphism in $\Nilp$ with
one-dimensional kernel $ks\subset n$. Let $x\in \mathfrak{g}^1\otimes m$ be a
solution to the Maurer-Cartan equation. Lift $x$ to an arbitrary
element $\hat{x}$ of $\mathfrak{g}^1\otimes n$ and let $p(\hat{x})\in
\mathfrak{g}^2$ be such that
$p(\hat{x})s=d\hat{x}+\frac{1}{2}[\hat{x},\hat{x}]$. Then clearly
$dp(\hat{x})=0$ and furthermore the cohomology class
$o(x)\overset{\text{def}}{=}\overline{p(\hat{x})}\in
  H^2(\mathfrak{g}^\bullet)$ does not depend  on the chosen
  lift $\hat{x}$ of~$x$.  It is easy to see that $o(x)=0$ if and only if $x$ can be lifted to an element
of $\MC(\mathfrak{g}^\bullet\otimes n)$. Consequently $o(x)$ is called the 
\emph{obstruction class} of $x$.

The \emph{obstruction space} $O(\mathfrak{g}^\bullet)$ is
the linear span in $H^2(\mathfrak{g}^2)$ of all $o(x)$ for all morphisms $(S,n)\rightarrow (R,m)$ with one-dimensional kernel and
all $x\in \MC(\mathfrak{g}^\bullet\otimes m)$ as above. 

Clearly $o(x)$ and hence $O(\mathfrak{g}^\bullet)$ is functorial under
DG-Lie algebra morphisms. It is well-known and easy to see that this functoriality extends to
$L_\infty$-morphisms.

Recall that the \emph{periodic cyclic complex}
$\CC_\bullet^{\text{per}}(A)$ of a $k$-algebra $A$ is obtained by inverting $u$ in
$\CC^-_\bullet(A)$. Its homology will be denoted by
$\HC_\bullet^{\text{per}}(A)$.  The following is the main result of
this section.
\begin{theorem} 
\label{ref-12.1-70} Let $(A,\bar{\eta})$ be a $d$-Calabi-Yau algebra. Then
the composition
\[
O(\mathfrak{D}^\bullet(A,\eta))\hookrightarrow H^2(\mathfrak{D}^\bullet(A,\eta))
\overset{\text{Thm \ref{ref-10.1-48}}}{\cong} \HC_{d-3}^-(A)\rightarrow \HC_{d-3}^{\text{per}}(A)
\]
is zero.
\end{theorem}
The proof depends on the following beautiful result by Daletskii  and  Tsygan
 \cite[Thm 1]{Tsygan1} (see also
  \cite{DaTs}).
\begin{theorem} The Lie action of $\mathfrak{C}^\bullet(A)$ on 
$\CC^-_\bullet(A)$ can be extended to a $u$-linear $L_\infty$-action $R$
of the DG-Lie algebra $(\mathfrak{C}^\bullet(A)[u,\epsilon],d+u\partial/\partial\epsilon)$, with
  $|\epsilon|=1$, $\epsilon^2=0$ and such that
\begin{align*}
\partial^1R(\gamma)&=d\gamma\\
\partial^2R(s\sigma,\gamma)&=L_\sigma \gamma\\
\partial^2R(s(\epsilon \sigma),\gamma)&=I_\sigma \gamma
\end{align*}
for $\sigma\in \mathfrak{C}^\bullet(A)$, $\gamma\in\CC^-_\bullet(A)$.
\end{theorem}
The statement about $\partial^2R(s(\epsilon \sigma),\gamma)$ does not occur in
 \cite[Thm 1]{Tsygan} but it follows easily from the proof.

\medskip

In the rest of this section $(A,\bar{\eta})$ is a $d$-Calabi-Yau algebra. 
\begin{lemma}
  There is a commutative diagram of complexes
\begin{equation}
\label{ref-12.1-71}
\xymatrix{
(\mathfrak{C}^\bullet(A)\ltimes \Sigma^{-d-1}\CC^-_\bullet(A))_{(0,\eta')}
\ar[rr]\ar[dr]_{\Psi} && (\mathfrak{C}^\bullet(A)[u,\epsilon]\ltimes \Sigma^{-d-1}\CC^-_\bullet(A))_{(0,\eta')}\ar[dl]^{\Psi'}\\
&
\Sigma^{-d+1}\CC^-_\bullet(A))&
}
\end{equation}
where 
\begin{itemize}
\item $\Psi$ was introduced in Theorem \ref{ref-10.1-48};
\item $\eta=s^{-d-1}\eta'$;
\item the horizontal map is a twist (see \S\ref{ref-11.2-62}) of the map obtained from the obvious inclusion of DG-Lie algebras
$(\mathfrak{C}^\bullet(A),d)\hookrightarrow (\mathfrak{C}^\bullet(A)[u,\epsilon],d+\partial/\partial\epsilon)$.
\item $\Psi'$ restricted to $\mathfrak{C}^\bullet(A)[u,\epsilon]$ is $u$-linear
and satisfies
\begin{equation}
\label{ref-12.2-72}
\begin{aligned}
\Psi'(\sigma)&=(-1)^{|\sigma|+1}I_\sigma \eta'\\
\Psi'(\epsilon\sigma)&=0
\end{aligned}
\end{equation}
for $\sigma\in \mathfrak{C}^\bullet(A)$.
\item $\Psi'$ restricted to $\Sigma^{-d-1}\CC_\bullet^-(A)$ is multiplication by $u$.
\end{itemize}
\def\qed{}\end{lemma}
\begin{proof}
The commutativity of the diagram is clear.  We only have to show that $\Psi'$ commutes
with the differential. For $\Psi'$ restricted to $\Sigma^{-d-1}\CC_\bullet^-(A)$
this is obvious. As far as the restriction of $\Psi'$
to $\mathfrak{C}^\bullet(A)[u,\epsilon]$ is concerned: the only non-trivial
case (given that $\Psi$ already commutes with the differential) is the evaluation on an element
of $\epsilon\mathfrak{g}$.

Using \eqref{ref-11.1-63} we find for $\sigma \in \mathfrak{C}^\bullet(A)$
\[
d_{(0,\eta')}(\epsilon\sigma)=(d(\epsilon \sigma),(-1)^{|g|}I_\sigma\eta')
\]
Given \eqref{ref-12.2-72} we have to show
\[
\Psi'(d_{(0,\eta')}(\epsilon\sigma)=0
\]
We compute
\begin{align*}
\Psi' (d_{(0,\eta')}(\epsilon \sigma))&=\Psi'(d(\epsilon \sigma),(-1)^{|\sigma|}I_\sigma\eta')\\
&=\Psi'(-\epsilon d\sigma+u\sigma,(-1)^{|\sigma|}I_\sigma\eta')\\
&=(-1)^{|\sigma|+1}uI_\sigma\eta'+(-1)^{|\sigma|}uI_\sigma\eta'\\
&=0\qed
\end{align*}
\def\qed{}\end{proof}

\begin{lemma} 
\label{ref-12.4-73} Consider $\Sigma^{-d+1}\CC^{\text{per}}_\bullet(A)$ as an abelian DG-Lie algeba. Then there
exists an $L_\infty$-morphism
\[
\Delta:\mathfrak{D}^\bullet(A,\eta)\rightarrow \Sigma^{-d+1}\CC^{\text{per}}_\bullet(A)
\]
such that the following diagram is commutative
\[
\xymatrix{
  H^\bullet(\mathfrak{D}^\bullet(A,\eta))\ar[drr]_{H^\bullet(\Delta)}\ar[rr]^-{H^\bullet(\Psi)}_-{}&&
H^\bullet(\Sigma^{-d+1}\CC^{-}_\bullet(A))\ar[d]^{\text{canonical}}\\
&&H^\bullet(\Sigma^{-d+1}\CC^{\text{per}}_\bullet(A))
}
\]
\end{lemma}
\begin{proof}
 To simplify the notations put
  $\mathfrak{g}^\bullet=\mathfrak{C}^\bullet(A)$, $V^-=\Sigma^{-d-1}\CC_\bullet^-(A)$,
$V^{\text{per}}=\Sigma^{-d-1}\CC_\bullet^{\text{per}}(A)$. 
Thus we get $L_\infty$-morphisms (see \S\ref{ref-11.1-61},\S\ref{ref-11.2-62})
\begin{multline}
\label{ref-12.3-74}
(\mathfrak{g}^\bullet\ltimes V^-)_{(0,\eta')}\rightarrow
(\mathfrak{g}^\bullet[u,\epsilon]\ltimes V^-)_{(0,\eta')}
\rightarrow
(\mathfrak{g}^\bullet[u,u^{-1},\epsilon]\ltimes V^{\text{per}})_{(0,\eta')}
\xleftarrow{c} (0\ltimes V^{\text{per}})_{(0,\eta')}
\\\cong 
V^{\text{per}}
\overset{\times u}\cong \Sigma^2 V^{\text{per}}
\end{multline}
Here $c$ goes in the wrong direction but it is easy to see that
$(\mathfrak{g}^\bullet[u,u^{-1},\epsilon],d+u\partial/\partial\epsilon)$ is
acyclic. Hence $c$ is an quasi-isomorphism.  This means that
there is an $L_\infty$-quasi-isomorphism $c'$ which goes in the
opposite direction and which inverts $c$ on the level of cohomology. Taking
the composition of everything we obtain an $L_\infty$-morphism 
\[
(\mathfrak{g}^\bullet\ltimes V^-)_{(0,\eta')}\mor \Sigma^2V^{\text{per}}
\]
which is the desired $\Delta$.

It remains to show that $\Delta$ and $\Psi$ are compatible on the
level of cohomology. This follows from the following commutative diagram 
whose upper row is a compressed version of  \eqref{ref-12.3-74}
and whose lower row we obtain from \eqref{ref-12.1-71}. 
\[
\xymatrix{
  (\mathfrak{g}^\bullet\ltimes V^-)_{(0,\eta')}\ar@/^2em/[rrrr]^{\Delta}\ar[r]\ar[d]_{\Psi}&(\mathfrak{g}^\bullet[u,u^{-1},\epsilon]\ltimes V^{\text{per}})_{(0,\eta')}\ar[d]^{\Psi'}\ar@<1ex>[rr]^-{c'}&\cong&V^{\text{per}}\ar@{=}[d]\ar@<1ex>[ll]^-{c}\ar[r]_{\times u}&\Sigma^2 V^{\text{per}}\ar@{=}[d]\\
\Sigma^2 V^-\ar[r]\ar@/_2em/[rrrr]_{\text{canonical}}&\Sigma^2 V^{\text{per}}&&V^{\text{per}}\ar[ll]^{\times u}\ar[r]^{\times u}&\Sigma^2 V^{\text{per}}
}\qed
\]
\def\qed{}\end{proof}
\begin{proof}[Proof of Theorem \ref{ref-12.1-70}] The theorem follows from Lemma \ref{ref-12.4-73}
together with the functoriality of  obstruction spaces under $L_\infty$-morphisms
and the fact that the obstruction space of an abelian Lie algebra is trivial. 
\end{proof}
\begin{corollary}
\label{ref-12.5-74}
If the map $\HC^{-}_{d-3}(A)\mor \HC^{\text{per}}_{d-3}(A)$ is injective then the
deformation theory of $A$ is unobstructed. 
\end{corollary}
This corollary applies for example in the case $d\le 3$ by the following
well-known lemma.
\begin{lemma}
\label{ref-12.6-75}
 $\HC^{-}_{n}(A)\mor \HC^{\text{per}}_{n}(A)$ is an isomorphism for $n\le 0$.
\end{lemma}
\begin{proof} There is an exact sequence
\[
\HC_{n-1}(A) \mor \HC^{-}_{n}(A)\mor \HC^{\text{per}}_{n}(A)\mor \HC_{n-2}(A)
\]
(e.g.\ \cite[Prop.\ 5.1.5]{Loday1}) where $\HC_\bullet(A)$ denotes
ordinary cyclic homology. The complex computing ordinary cyclic
homology is concentrated in homological degrees $\ge 0$. Hence
$\HC_n(A)=0$ for $n<0$. This finishes the proof.
\end{proof}
\begin{remark} Many 3-dimensional Calabi-Yau algebras are obtained from superpotentials
  (see \cite{Bocklandt,VdBsuper}). 
For those it is is not very surprising that the 
deformation theory is unobstructed (the deformations come from deforming the superpotential). However
there are examples of 3-dimensional Calabi-Yau algebras which are not obtained
from superpotentials. See e.g.\ \cite{davison}. Simple examples are given by 
3-dimensional smooth commutative Calabi-Yau algebras  with no exact volume form.
\end{remark}
\appendix

\section{A technical result on BV-algebras}
\label{ref-A-76}
Recall that a DG-BV-algebra is a quadruple $(\mathfrak{g}^\bullet,d,\Delta,\cup)$ where $(\mathfrak{g}^\bullet,d)$
is a complex, $\cup$ is a commutative, associative product of degree\footnote{As always our
grading conventions are such that Lie brackets have degree zero.} $1$ on $\mathfrak{g}^\bullet$ compatible
with $d$, $\Delta$ is a differential of degree $-1$, $(\mathfrak{g}^\bullet,d,[-,-])$ is a
DG-Lie algebra with $[-,-]$ defined by:
\[
[g,h]=(-1)^{|g|+1}(\Delta(g\cup h)-\Delta g\cup h-(-1)^{|g|+1}g\cup \Delta h)
\]
and $\cup$, $[-,-]$ are related by the Leibniz rule:
\[
[g,h_1\cup h_2]=[g,h_1]\cup h_2+(-1)^{|g|(|h_1|+1)}h_1\cup [g,h_2]
\]
It is shown in \cite{KKP,Terilla} that if $\mathfrak{h}^\bullet$ is a DG-BV-algebra then 
$(\mathfrak{h}^\bullet((u)),d+u\Delta)$ is homotopy abelian. The same proof works
for $u\mathfrak{h}^\bullet[[u]],d+u\Delta)$ but not for $(\mathfrak{h}^\bullet[[u]],d+u\Delta)$.
Our aim in this section is to make $(\mathfrak{h}^\bullet[[u]],d+u\Delta)$ as ``commutative as possible''
(see Proposition \ref{ref-A.4-81} below) by making at least its sub-DG-Lie algebra
$(u\mathfrak{h}^\bullet[[u]],d+u\Delta)$ abelian. This is not completely straightforward
since in order to do this we have to twist the action of $\mathfrak{h}^\bullet$ on $u\mathfrak{h}^\bullet[[u]]$.

The fact that $(\mathfrak{h}^\bullet((u)),d+u\Delta)$ and
$(u\mathfrak{h}^\bullet[[u]],d+u\Delta)$ are homotopy abelian is in fact a special
case of a general result in \cite{ShTa}.
For the benefit of the reader we repeat the proof of this
result. Afterwards we will reuse the proof to treat
$(\mathfrak{h}^\bullet[[u]],d+u\Delta)$.

It is convenient to use the following adhoc definition.
\begin{definition}
A $\BVm$ algebra is a DG-Lie algebra $\mathfrak{g}^\bullet$ equipped with a commutative, associative
product $\cup$ of degree $-1$, compatible with $d$,  such that 
\begin{equation}
\label{ref-A.1-77}
[g,h]=(-1)^{|g|+1}(d(g\cup h)-dg\cup h-(-1)^{|g|+1}g\cup dh)
\end{equation}
and
\begin{equation}
\label{ref-A.2-78}
[g,h_1\cup h_2]=[g,h_1]\cup h_2+(-1)^{|g|(|h_1|+1)}h_1\cup [g,h_2]
\end{equation}
\end{definition}
\begin{lemma} \cite{ShTa}
\label{ref-A.2-79}
Let ${\mathfrak{g}^\bullet}$ be a $\BVm$-algebra and let 
${\mathfrak{a}^\bullet}$ be the same as ${\mathfrak{g}^\bullet}$ but with the Lie bracket set to
zero. Then there is a $L_\infty$-morphism $\psi:{\mathfrak{g}^\bullet}\mor {\mathfrak{a}^\bullet}$
such $\partial^1\psi$ is the identity. In other words ${\mathfrak{g}^\bullet}$ is
homotopy abelian.
\end{lemma}
\begin{example} Let $({\mathfrak{h}^\bullet},d,\Delta,\cup)$ be a DG-BV-algebra. Then
  $(u{\mathfrak{h}^\bullet}[[u]],d+u\Delta,[-,-],u^{-1}(-\cup-))$ is a
  $\BVm$-algebra and hence by the previous lemma 
  $(u{\mathfrak{h}^\bullet}[[u]],d+u\Delta)$ is homotopy abelian. The same reasoning applies
to $(\mathfrak{h}^\bullet((u)),d+u\Delta)$.
\end{example}

\begin{proof}[Proof of Lemma \ref{ref-A.2-79}]
Put $V^\bullet=\Sigma {\frak{g}^\bullet}$. The coderivation $Q$ on $S^cV^\bullet$ corresponding to the
DG-Lie structure is given by
\[
\partial^1 Q:V^\bullet\mor V^\bullet:sg\mapsto -s\,dg
\]
\[
\partial^2 Q:S^2 V^\bullet\mor V^\bullet:(sg,sh)\mapsto (-1)^{|g|}s[g,h]
\]
and all other $\partial^n Q$ are zero.

For simplicity of notation we put
\[
sg_1\cdot sg_2 \cdots  sg_n=s(g_1\cup\cdots \cup g_n)
\]
From \eqref{ref-A.1-77}\eqref{ref-A.2-78} we obtain:
\begin{equation}
\label{ref-A.3-80}
\begin{aligned}
\partial^1 Q(v_1\cdot v_2\cdots v_n)&=
\sum_i \epsilon_i \partial^1 Q(v_i)v_1\cdots \hat{v}_i\cdots v_n
+\sum_{i<j} \epsilon_{i,j}\partial^2 Q(v_i,v_j)v_1\cdots \hat{v}_i\cdots \hat{v}_j\cdots v_n
\end{aligned}
\end{equation}
where the signs are determined by
\begin{align*}
v_1\cdot v_2\cdots v_n&=\epsilon_i v_i\cdot v_1\cdots \hat{v}_i\cdots v_n\\
&=\epsilon_{i,j}v_i\cdot v_j\cdot v_1\cdots \hat{v}_i\cdots \hat{v}_j\cdots v_n
\end{align*}
Consider $\partial^1 Q$ as a coderivation of $S^cV^\bullet$
and let $\psi:S^cV^\bullet\mor S^cV^\bullet$ be the coalgebra automorphism determined 
by
\[
\partial^n\psi(v_1,\ldots,v_n)=v_1\cdot v_2\cdots v_n
\]
 Then \eqref{ref-A.3-80} becomes
\[
\partial^1 Q\circ \psi=\psi\circ Q
\]
which finishes the proof.
\end{proof}
\begin{proposition} 
\label{ref-A.4-81} Let $({\mathfrak{h}^\bullet},d,\Delta,\cup)$ be a
DG-BV-algebra. Let ${\mathfrak{a}^\bullet}$ be the graded vector space
 $u{\mathfrak{h}^\bullet}[[u]]$. The following operation
\begin{equation}
\label{ref-A.4-82}
h\star a=[h,a]+(-1)^{|h|+1}\Delta(h)\cup a
\end{equation}
for $h\in {\mathfrak{h}^\bullet}$, $a\in {\mathfrak{a}^\bullet}$ makes ${\mathfrak{a}^\bullet}$
into a graded ${\mathfrak{h}^\bullet}$-representation.
Furthermore $d+u\Delta$ defines a derivation on the Lie algebra
${\mathfrak{h}^\bullet}\ltimes {\mathfrak{a}^\bullet}$ and finally there is an
$L_\infty$-isomorphism
\[
\phi:{\mathfrak{h}^\bullet}[[u]]\mor ({\mathfrak{h}^\bullet}\ltimes {\mathfrak{a}^\bullet},d+u\Delta)
\]
such that $\partial^1\phi$ is the identity.
\end{proposition}
\begin{proof}
 In the proof below we identify the underlying
vector spaces of $\mathfrak{h}^\bullet[[u]]$ and $\mathfrak{h}^\bullet\ltimes\mathfrak{a}$ in the obvious way.
The fact that \eqref{ref-A.4-82} defines indeed a
representation as well as compatibility with differentials is an easy direct verification:
Now put $V^\bullet=\Sigma {\mathfrak{a}^\bullet}$, $W^\bullet=\Sigma {\mathfrak{h}^\bullet}$.
Let $Q$ be the coderivation on $S^c(W^\bullet\oplus V^\bullet)$ corresponding to 
${\mathfrak{h}^\bullet}[[u]]$. We observe
that $\partial^1Q{\mid} W^\bullet=\partial^1 Q_1+\partial^2 Q_2$ where
$\partial^1 Q_1=-d$ and $\partial^1 Q_2=-u\Delta$. Let $Q'$ be the
coderivation on $S^c(W^\bullet\oplus V^\bullet)$ corresponding to 
$({\mathfrak{h}^\bullet}\ltimes {\mathfrak{a}^\bullet},d+u\Delta)$. We have $\partial^1Q'=\partial^1 Q$. Furthermore
\begin{align*}
\partial^2 Q'(w_1,w_2)&=\partial^2 Q(w_1,w_2)&&\text{for $w_1,w_2\in W^\bullet$}\\
\partial^2 Q'(v_1,v_2)&=0&&\text{for $v_1,v_2\in V^\bullet$}
\end{align*}
and for $h\in {\mathfrak{h}^\bullet}$, $a\in {\mathfrak{a}^\bullet}$
\begin{align*}
\partial^2 Q'(sh,sa)&=(-1)^{|h|}s(h\star a)\\
&=(-1)^{|h|}s[h,a]-s(\Delta h\cup a)\\
&=\partial^2 Q(sh,sa)+\partial^1 Q_2(sh)\cdot sa
\end{align*}
where as above $x\cdot y=u^{-1}(x\cup y)$.
In other words
\begin{align}
\label{ref-A.5-83}
\partial^2 Q'(w,v)&=\partial^2 Q(w,v)+\partial^1 Q_2(w)\cdot v
&&\text{for $w\in W^\bullet$, $v\in V^\bullet$}
\end{align}
We now construct the desired $L_\infty$-morphism. By definition $\partial^n \psi=\Id$ for $n=1$. For $n>1$, $i\ge 1$, $w_1,\ldots,w_i\in W^\bullet$, $v_1,\ldots,v_j\in V^\bullet$
we put
\[
\partial^n\psi(w_1,\ldots,w_i,v_1,\ldots,v_j)=0
\]
and
\[
\partial^n\psi(v_1,\ldots,v_j)=v_1\cdot v_2\cdots v_n
\]
We now verify
\[
\psi \circ Q=Q'\circ \psi
\]
We must evaluate both sides on $S^i W^\bullet\otimes S^j V^\bullet$. If $i=0$ then 
the desired equality follows from the proof of Lemma \ref{ref-A.2-79}.
If $i> 2$ then both sides are zero so this case is trivial as well.
If $i=2$ then both sides are zero unless $j=0$ in which case we
reduce to $\partial^2 Q{|} S^2W^\bullet=\partial^2 Q'{|} S^2W^\bullet$.

We concentrate on the case $i=1$. We find
\[
(Q'\circ \psi)(w_1,v_1,\ldots,v_j)=\partial^2Q'(w_1,v_1\cdot v_2\cdots v_n)
\]
and
\begin{align*}
(\psi\circ Q)(w_1,v_1,\ldots,v_j)&=\partial^1Q_2(w_1)\cdot v_1\cdots v_j+\sum_l \pm\partial^2Q(w_1,v_l)\cdot v_1
\cdots \hat{v}_l\cdots v_j\\
&=\partial^1Q_2(w_1)\cdot v_1\cdots v_j+\partial^2Q(w_1,v_1\cdot v_2\cdots v_j)
\end{align*}
We conclude by \eqref{ref-A.5-83}.
\end{proof}
\def\cprime{$'$} \def\cprime{$'$} \def\cprime{$'$}
\ifx\undefined\bysame
\newcommand{\bysame}{\leavevmode\hbox to3em{\hrulefill}\,}
\fi

\end{document}